\documentclass{ws-ijnt}
\usepackage[super]{cite}
\usepackage{xcolor}
\usepackage[verbose,hypertexnames=false]{hyperref}
\hypersetup{colorlinks=false,allbordercolors=blue,pdfborderstyle={/S/U/W 1}}
\usepackage{mathtools}
\usepackage{tikz, float, tikzscale}
\usepackage{pgfplots}
\pgfplotsset{compat=1.17}
\usepackage{subcaption}
\usepackage{caption}

\definecolor{mycolour1}{RGB}{153,153,153}
\newlength{\mytikzwidth}
\tikzset{
	declare function={
		h(\x,\y)=\mytikzwidth*(1/\x+(\y-1)/((((\y-1)*\y)-1))/(((\x-1)*\x));
	}
}
\newlength{\mytikzwidtj}
\tikzset{
	declare function={
		f(\x,\y)=\mytikzwidtj*(1/\x+(\y-1)/((((\y-1)*\y)-1))/(((\x-1)*\x));
	}
}

\newcommand{\interval}[5]{
	\fill [mycolour1] ({-.01*(2*(#2-3)+1)*2\textheight},{-1*#3-.1}) rectangle ({-.01*(2*(#1-2))*2\textheight},{-1*#3+.1});
	\ifnum#4=1
	\node at ({-.01*(2*(#1-2))*2\textheight},{-1*#3+.3}) {$\langle #1$};
	\fi
	\draw ({-.01*(2*(#1-2))*2\textheight},{-1*#3-.15})--({-.01*(2*(#1-2))*2\textheight},{-1*#3+.15});
	\ifnum#5=1
	\node at ({-.01*(2*(#2-3)+1)*2\textheight},{-1*#3+.3}) {$#2\rangle$};
	\fi
	\draw ({-.01*(2*(#2-3)+1)*2\textheight},{-1*#3-.15})--({-.01*(2*(#2-3)+1)*2\textheight},{-1*#3+.15});
}

\begin{document}
	
	\markboth{Maiken Gravgaard, Ying Wai Lee}{Decomposition of real numbers into sums of {L\"uroth} sets}
	
	%%%%%%%%%%%%%%%%%%%%% Publisher's Area please ignore %%%%%%%%%%%%%%
	\catchline{}{}{}{}{}
	%%%%%%%%%%%%%%%%%%%%%%%%%%%%%%%%%%%%%%%%%%%%%%%%%%%%%%%%%%%%%%%%%%%
	
	\title{Decomposition of real numbers into sums of {L\"uroth} sets}

	\author{Maiken Gravgaard\footnote{
			Department of Mathematics, Aarhus University, Ny Munkegade 118,
			8000 Aarhus, Denmark}}
	
	\address{Department of Mathematics, Aarhus University, Ny Munkegade 118\\
		Aarhus, 8000,Denmark\,\\
		\email{maiken@math.dk} }
	
	\author{Ying Wai Lee}
	
	\address{Department of Mathematics, University of York\\
		York, YO10 5DD, United Kingdom\\
		yingwai.lee@york.au.dk }
	
	\maketitle
	
	\begin{history}
		\received{(14/06/2025)}
		\revised{(18/11/2025)}
		\accepted{(01/02/2026)}
		%\published{(Day Month Year)}
	\end{history}
	
	\begin{abstract}
		We study the decomposition of real numbers into sums of L\"uroth sets, which are defined by numbers whose L\"uroth expansions have prescribed digit constraints. We establish several results on the congruence modulo 1 of sums of L\"uroth sets, including summands with digits bounded above, below, and combinations of the two. We also analyse the Hausdorff dimension of L\"uroth sets and their sums. The results extend classical findings on continued fractions to L\"uroth expansions.
	\end{abstract}
	
	
	
	\section{Introduction}\label{sectintro}
	Continued fractions and dynamical systems are fundamental concepts in mathematics. Connections between number theory and fractal geometry can be studied via decomposing numbers as sums of elements from specific sets, such as those defined by restricted digit representations.
	
	Let $A,B\subseteq \mathbb{R}$ be two subsets of the real numbers. The sumset of $A$ and $B$, denoted $A+B$, is defined as the Minkowski sum of $A$ and $B$:
	\begin{align*}
		A+B
		\coloneqq
		\left\{a+b : a\in A \text{ and } b\in B\right\}.
	\end{align*}
	The above definition extends naturally to the sumset of multiple subsets of $\mathbb{R}$. For any $n\in\mathbb{N}$, the $n$-fold sumset of $A$ with itself, denoted $nA$, is defined as:
	\begin{align*}
		n A\coloneqq \left\{a_1+a_2+\dots+a_n:a_1,a_2,\ldots,a_n\in A\right\}.
	\end{align*}
	We adopt the following notation for congruence modulo 1: $A\equiv B\mod1$, if $A+\mathbb{Z}=B+\mathbb{Z}$; that is every element of $B$ can be written as the sum of an element of $A$ and an integer, and vice versa. Otherwise, we write $A\not\equiv B\mod1$.
	
	Let $C_{1/3}$ be the middle-thirds Cantor set. A classical result~\cite{randolph1940, steinhaus1920} states that 
	\begin{equation}\label{middlethird}
		C_{1/3}+C_{1/3} = [0,2].
	\end{equation}
	Although $C_{1/3}$ is a fractal with Lebesgue measure $0$ and Hausdorff dimension $\log{2}/\log{3}$, two of them sum up to an interval of length 2. The middle-thirds Cantor set and the result \eqref{middlethird} can be understood from a number theory point of view. Notice that
	\begin{align*}
		C_{1/3}
		=\left\{\sum_{n=1}^{\infty} \frac{a_n}{3^n}\in[0,1]:a_n\in\{0,2\}\text{ for all $n\in\mathbb{N}$}\right\},
	\end{align*}
	that is, $C_{1/3}$ is the set of numbers whose ternary expansion does not contain any ones. With this interpretation, \eqref{middlethird} implies that $C_{1/3}+C_{1/3} \equiv\mathbb{R}\mod1$. In other words, every real number can be written as a sum of an integer and two numbers whose ternary expansions do not contain any ones.
	
	This connection between base expansions and additive structure illustrates how fractal sets can possess rich number-theoretic properties. In particular, it motivates the study of sums of fractal sets and their role in problems concerning the decomposition of real numbers into sums of elements with restricted digit representations.
	
	\subsection{Results for Continued Fractions}\label{subsectrfcf}
	With the above example in mind, one could choose a different representation for the real numbers than ternary expansions, for example, continued fractions.
	Define, for any $k\in\mathbb{N}$,
	\begin{align*}
		F_{\leq k}
		\coloneqq
		\left\{[a_1,a_2,a_3,\ldots] : 1\leq a_n \leq k \text{ for all } n\in\mathbb{N}\right\},
	\end{align*}
	where $[a_1,a_2,a_3,\ldots]$ denotes the continued fraction for the partial quotients $(a_n)_{n\in\mathbb{N}}$, that is,
	\begin{align*}
		[a_1,a_2,a_3,\ldots]
		\coloneqq\frac{1}{\displaystyle a_1+\frac{1}{\displaystyle a_2+\frac{1}{a_3+\cdots}}}.
	\end{align*}
	
	Sumsets that are induced by continued fractions are studied by various mathematicians. In a work of Hall~\cite[Theorem 3.1]{hall1947}, it is proved that every real number can be expressed as a sum of an integer and two continued fractions with partial quotients not greater than 4; in other words,
	\begin{align*}
		F_{\leq 4}+F_{\leq 4} \equiv \mathbb{R} \mod 1.
	\end{align*}
	In particular, every real number can be decomposed into the sum of two badly approximable numbers. Several papers have since appeared that extend this classical result. Below we mention the most relevant for our purposes. By employing refined versions of the techniques used by Hall, Hlavka~\cite[Theorems 6 \& 9]{Hlavka1975} proved that
	\begin{align*}
		F_{\leq 3}+F_{\leq 4} \equiv \mathbb{R} \mod 1; &&
		F_{\leq 2}+F_{\leq 7} \equiv \mathbb{R} \mod 1.
	\end{align*}
	Hlavka~\cite[Theorems 11]{Hlavka1975} also extended the study into the sums of multiple sets:
	\begin{align*}
		&F_{\leq 2}+F_{\leq 3}+F_{\leq 3} \equiv \mathbb{R} \mod 1, \\
		&F_{\leq 2}+F_{\leq 2}+F_{\leq 4} \equiv \mathbb{R} \mod 1, \\
		&F_{\leq 2}+F_{\leq 2}+F_{\leq 2}+F_{\leq 2} \equiv \mathbb{R} \mod 1.
	\end{align*}
	Hlavka~\cite{Hlavka1975} also proved that $F_{\leq 2}+F_{\leq 4} \not\equiv \mathbb{R} \mod 1$. Divi\v{s}~\cite[Theorems 1 \& 2]{divis1973} proved that some of these results of Hlavka are optimal, in the sense that
	\begin{align*}
		F_{\leq 3}+F_{\leq 3}\not\equiv \mathbb{R} \mod 1; &&
		F_{\leq 2}+F_{\leq 2}+F_{\leq 2} \not\equiv \mathbb{R} \mod 1.
	\end{align*}
	Astels~\cite[Theorem 1.1]{astels2001sums}~\cite[Theorem 1]{astels2002sums} later invented another tool to prove that
	\begin{align*}
		F_{\leq 2}+F_{\leq 5} \equiv \mathbb{R} \mod 1; &&
		F_{\leq 2}+F_{\leq 2}+F_{\leq 3} \equiv \mathbb{R} \mod 1.
	\end{align*}
	
	The sets in which the partial quotients are bounded from below are studied by Cusick among others. Define for any $k\in\mathbb{N}$,
	\begin{align*}
		F_{\geq k}
		\coloneqq
		\left\{[a_1,a_2,a_3,\ldots] : a_n \geq k \text{ for all } n\in\mathbb{N}\right\}\cup F'_{\geq k},
	\end{align*}
	where $F'_{\geq k}$ is the set consisting of 0 and all continued fractions of finite lengths with partial quotients at least $k$; that is, continued fractions of finite lengths are also allowed in $F_{\geq k}$. Cusick~\cite[Theorem 1]{cusick1971sums} proved that
	\begin{align*}
		F_{\geq 2}+F_{\geq 2}\equiv\mathbb{R} \mod 1.
	\end{align*}
	Cusick--Lee~\cite{lee1971sums} extend this result by showing that for any $k\in\mathbb{N}$,
	\begin{align}\label{eq: ksk}
		k F_{\geq k} \equiv \mathbb{R} \mod 1.
	\end{align}
	
	Good~\cite[Theorem 2]{Good_1941} proved that for any $k\in\mathbb{N}$, if $k\geq20$ then
	\begin{align*}
		\frac{1}{2}+\frac{1}{2\log{(k+2)}}
		<\dim{F_{\geq k}}
		<\frac{1}{2}+\frac{\log{\log{(k-1)}}}{2\log{(k-1)}},
	\end{align*}
	where $\dim$ denotes the Hausdorff dimension. Further refinements of Good's bounds, along with an analysis of the Hausdorff dimension of the sets $F_{\leq k}$, were recently shown by Das--Fishman--Simmons--Urbański in~\cite{Das_2023}. Shulga~\cite[Theorem 1.2]{Shulga2024CusickCantor} invented an algorithm and proved that for any $k\in\mathbb{N}$, if $k\geq2$ then
	\begin{align*}
		\left[0,\frac{1}{k-1}\right]\subseteq F_{\geq k}+F_{\geq k};
	\end{align*}
	as well as some other results on sets of the form $F_{\geq k_1}+F_{\geq k_2}$. In a recent work, Gayfulin--Nesharim~\cite{gayfulin2025realnumbersumreal} derive more results for the summands in Shulga's algorithm.
	
	Difference and product sets derived from continued fraction expansions are also studied by Astels~\cite{astels2000,astels2001sums,astels2002sums}. However, we shift our focus to the L\"uroth expansions and seek to establish analogous results for their corresponding sumsets.
	
	\subsection{L\"uroth Expansion}\label{subsectle}
	In this subsection, we explain what the L\"uroth expansion (or the L\"uroth representation of numbers) is, and why its fractal sumsets are natural to study, due to the dynamical systems involved. For any $x\in(0,1]$, there exists a unique sequence $(d_n)_{n\in\mathbb{N}}$ of positive integers not equal to 1, referred to as digits, such that 
	\begin{align*}
		x
		&= [d_1, d_2, d_3, \dots] \\
		&\coloneqq\sum_{n=1}^{+\infty}\frac{1}{d_n\prod_{j=1}^{n-1}d_j(d_j-1)} \\
		&=\frac{1}{d_1}+\frac{1}{d_1(d_1-1)d_2}+\frac{1}{d_1(d_1-1)d_2(d_2-1)d_3}+\cdots,
	\end{align*}
	where for any $n\in\mathbb{N}$, $d_n\in\mathbb{N}\setminus\{1\}$. The notation $[d_1, d_2, d_3,\ldots]$ denotes the L\"uroth representation when the brackets contain $d$'s instead of $a$'s, distinguishing it from the continued fraction representation $[a_1, a_2, a_3, \ldots]$.
	
	The middle-thirds Cantor set $C_{1/3}$ has an invariant property under the times-3 map (see Figure~\ref{fig: base-3}). The map shifts base-3 digits left (a symbolic shift), so the map restricted to $C_{1/3}$ is a dynamical system conjugate to the full shift on two symbols, meaning it mimics the dynamics of all possible binary sequences (that is, using symbols 0 and 2). Thus, it preserves the structure of $C_{1/3}$, making $C_{1/3}$ invariant under the map and a rich object in dynamical systems theory.
	
	Just as the times-3 map preserves the middle-thirds Cantor set, a more complicated map preserves continued fraction expansions with restricted partial quotients. The Gauss map (also known as the Gauss--Kuzmin--Wirsing operator) generates the continued fraction partial quotients and dynamically shifts the continued fraction expansions. Hence, the sets $F_{\leq k}$ and $F_{\geq k}$ are both invariant under this map. The Gauss map restricted to $F_{\leq k}$ (or $F_{\geq k}$) models a symbolic dynamical system on the finite alphabet $\{1, 2, \ldots, k\}$ (or infinite alphabet $\{k, k+1, \ldots\}$). The set $F_{\leq k}$ is also well-studied in Diophantine approximation (see~\cite[Theorem 1.4]{beresnevich2016metricdiophantineapproximationaspects}) and Fourier analysis (see~\cite{kaufman1980continued,queffelec2003analyse,sahlsten2024fourier}).
	
	%The above works!!
	
	\begin{figure}
		\centering
		\begin{minipage}{.5\textwidth}
			\centering
			\begin{tikzpicture}
				\begin{axis}[
					axis lines = left,
					axis x line=center,
					axis y line=center,
					width=1.1\textwidth,
					height=1.1\textwidth,
					xmin=-.075,
					xmax=1.15,
					ymin=-.075,
					ymax=1.15,
					xtick={1,1/3,2/3},
					xticklabels={1,$1/3$,$2/3$},
					ytick={1},
					xlabel = \(x\),
					ylabel = \(3x\mod1\),
					axis equal,
					]
					\foreach \d in {0,1,2} {
						\addplot[domain={\d/3:(\d+1)/3}, samples=2] {3*x-\d};
					}
					\draw[dotted] (  0,1) -- (  1,1); 
					\draw[dotted] (1/3,0) -- (1/3,1); 
					\draw[dotted] (2/3,0) -- (2/3,1); 
					\draw[dotted] (  1,0) -- (  1,1); 
				\end{axis}
			\end{tikzpicture}
			\caption{Times-3 Map}
			\label{fig: base-3}
		\end{minipage}%
		\begin{minipage}{.5\textwidth}
			\centering
			\begin{tikzpicture}
				\begin{axis}[
					axis lines = left,
					axis x line=center,
					axis y line=center,
					width=1.1\textwidth,
					height=1.1\textwidth,
					xmin=-.075,
					xmax=1.15,
					ymin=-.075,
					ymax=1.15,
					xtick={1/3,1/2,1},
					xticklabels={$1/3$,$1/2$,1},
					ytick={1},
					xlabel = \(x\),
					ylabel = \(1/x\mod1\),
					axis equal,
					]
					\draw[dotted] (  0,1) -- (  1,1); 
					\draw[dotted] (  1,0) -- (  1,1); 
					\foreach \d in {1, 2, ..., 30} {
						\addplot[smooth, domain={1/(\d+1):1/\d}, samples=300] {1/x-\d};
					}
					\draw[dotted] (1/2,0) -- (1/2,1); 
					\draw[dotted] (1/3,0) -- (1/3,1); 
					\draw[dotted] (1/4,0) -- (1/4,1); 
					\draw[dotted] (1/5,0) -- (1/5,1); 
					\draw[dotted] (1/6,0) -- (1/6,1); 
					\draw[dotted] (1/7,0) -- (1/7,1); 
					\draw[dotted] (1/8,0) -- (1/8,1); 
					\draw[dotted] (1/9,0) -- (1/9,1); 
				\end{axis}
			\end{tikzpicture}
			\caption{Gauss--Kuzmin--Wirsing Operator}
			\label{fig: Gauss}
		\end{minipage}
	\end{figure}
	
	While the times-3 map consists of finitely many piecewise linear segments and the Gauss map consists of infinitely many piecewise non-linear segments (see Figure~\ref{fig: Gauss}), both play fundamental roles in dynamical systems and number theory. Each generates fractal structures in the interval $[0,1]$ through their invariant sets, and the sumsets of these fractals have been extensively studied. It is therefore natural to consider their synthesis, known as the \emph{linearised Gauss map} or the \emph{L\"uroth map}, combining linearity and infinite segments (see Figure~\ref{fig: Luroth map}). Like its predecessors, it exhibits rich arithmetic properties and gives rise to fractals in $[0,1]$ whose sumsets display intriguing behaviour. 
	
	\begin{figure}
		\centering
		\begin{minipage}{.5\textwidth}
			\centering
			\begin{tikzpicture}
				\begin{axis}[
					axis lines = left,
					axis x line=center,
					axis y line=center,
					width=1.1\textwidth,
					height=1.1\textwidth,
					xmin=-.075,
					xmax=1.15,
					ymin=-.075,
					ymax=1.15,
					xtick={1,1/2,1/3},
					xticklabels={1,$1/2$,$1/3$},
					ytick={1},
					xlabel = \(x\),
					ylabel = \(T(x)\),
					axis equal,
					]
					\draw[dotted] (  0,1) -- (  1,1);
					\foreach \d in {2, 3, ..., 31} {
						\addplot[domain={1/\d:1/(\d-1)}, samples=2] {(\d-1)*(\d*x-1)};
					}
					\draw[dotted] (1/2,0) -- (1/2,1); 
					\draw[dotted] (1/3,0) -- (1/3,1); 
					\draw[dotted] (1/4,0) -- (1/4,1); 
					\draw[dotted] (1/5,0) -- (1/5,1); 
					\draw[dotted] (1/6,0) -- (1/6,1); 
					\draw[dotted] (1/7,0) -- (1/7,1); 
					\draw[dotted] (1/8,0) -- (1/8,1); 
					\draw[dotted] (1/9,0) -- (1/9,1); 
					\draw[dotted] (  1,0) -- (  1,1); 
				\end{axis}
			\end{tikzpicture}
			\caption{L\"uroth Map $T:(0,1]\to(0,1]$}
			\label{fig: Luroth map}
		\end{minipage}%
		\begin{minipage}{.5\textwidth}
			\centering
			\begin{tikzpicture}
				\begin{axis}[
					axis lines = left,
					axis x line=center,
					axis y line=center,
					width=1.1\textwidth,
					height=1.1\textwidth,
					xmin=-.075,
					xmax=1.15,
					ymin=-.075,
					ymax=1.15,
					xtick={1,1/2,1/3},
					xticklabels={1,$1/2$,$1/3$},
					ytick={1},
					xlabel = \(x\),
					ylabel = \(1-T(x)\),
					axis equal,
					]
					\draw[dotted] (  0,1) -- (  1,1); 
					\foreach \d in {2, 3, ..., 31} {
						\addplot[domain={1/\d:1/(\d-1)}, samples=2] {1-(\d-1)*(\d*x-1)};
					}
					\draw[dotted] (1/2,0) -- (1/2,1); 
					\draw[dotted] (1/3,0) -- (1/3,1); 
					\draw[dotted] (1/4,0) -- (1/4,1); 
					\draw[dotted] (1/5,0) -- (1/5,1); 
					\draw[dotted] (1/6,0) -- (1/6,1); 
					\draw[dotted] (1/7,0) -- (1/7,1); 
					\draw[dotted] (1/8,0) -- (1/8,1); 
					\draw[dotted] (1/9,0) -- (1/9,1); 
					\draw[dotted] (  1,0) -- (  1,1); 
				\end{axis}
			\end{tikzpicture}
			\caption{Alternating L\"uroth Map}
			\label{fig: A-Luroth map}
		\end{minipage}
	\end{figure}
	
	The L\"uroth map, denoted $T:(0,1]\to(0,1]$, is defined by, for any $x\in(0,1]$,
	\begin{align*}
		T(x)\coloneqq\left\lfloor\frac{1}{x}\right\rfloor\left(\left\lfloor\frac{1}{x}+1\right\rfloor x-1\right).
	\end{align*}
	The L\"uroth map, as well as the other two mentioned maps above, can generate the digits or the partial quotients of numbers (see~\cite{dajanikraaikamp2002ergodictheoryofnumbers}). Recent studies on the L\"uroth map and L\"uroth expansion include the metric theory of Diophantine approximation~\cite{cao2013efficiency,lee2025lurothexpansionsdiophantineapproximation,tan2021approximation,tan2021dimension}.
	
	Define, for any $\mathcal{A}\subset\mathbb{N}\setminus\{1\}$, the L\"uroth set $L_{\mathcal{A}}$ to be the set of numbers whose digits in their L\"uroth representations are restricted to $\mathcal{A}$; that is
	\begin{align*}
		L_{\mathcal{A}}
		\coloneqq \{[d_1, d_2, d_3, \dots] : d_n\in \mathcal{A} \text{ for all } n\in\mathbb{N}\}.
	\end{align*}
	These kinds of sets have also been considered in previous work~\cite{brouwer2024cantor,lee2024explicitupperboundsdecay}. For any $\mathcal{A}\subseteq\mathbb{N}\setminus\{1\}$, the L\"uroth set $L_{\mathcal{A}}$ is invariant under the L\"uroth map $T$.
	
	Our focus is on the following sets, as we aim to draw analogous conclusions to the continued fraction sets we discussed. Define for any $k\in\mathbb{N}\setminus\{1\}$, L\"uroth sets of digits bounded from above $L_{\leq k}$ and respectively below $L_{\geq k}$ by:
	\begin{align*}
		L_{\leq k}
		\coloneqq L_{\{2,3,\ldots,k\}}, &&
		L_{\geq k}
		\coloneqq L_{\{k,k+1,\ldots\}},
	\end{align*}
	and for any $N_1,N_2\in\mathbb{N}\setminus\{1\}$,
	\begin{align*}
		L_{N_1,N_2}\coloneqq L_{\geq N_1} \cap L_{\leq N_2}.
	\end{align*}
	
	In the work of Brouwer~\cite[(6.1)]{brouwer2024cantor}, an attempt is made to apply Hall's tool~\cite[Theorem 2.2]{hall1947} in establishing the following statement on the sum of L\"uroth sets:
	\begin{align*}
		L_{\leq 4}+L_{\leq 4} \equiv \mathbb{R} \mod 1.
	\end{align*}
	The idea of a proof is outlined but not fully completed in~\cite{brouwer2024cantor}. We verify that the above claim is true by applying a tool developed by Hlavka~\cite[Theorem 3]{Hlavka1975} (the tools are summarised in Section~\ref{section: GCS}), alongside this, we present further results on the sums of L\"uroth sets in this paper.
	
	%Above seems to work:)
	
	\section{Main Results}\label{section: results}
	In this paper, we provide two primary types of results concerning L\"uroth sets. The first type is on the equivalence modulo 1 of sums of L\"uroth sets, as presented in Section~\ref{section: congurence}. The second type is on the Hausdorff dimension of L\"uroth sets and their sums, as presented in Section~\ref{section: dimension}.
	
	\subsection{Congruence}\label{section: congurence}
	Theorem~\ref{thm: 1} can be seen as an analogue to the results of Hall, Hlavka, and Astels.
	\begin{theorem}\label{thm: 1}
		We have
		\begin{align*}
			L_{\leq 3}+L_{\leq 5} &\equiv\mathbb{R}\mod1; \\
			L_{\leq 4}+L_{\leq 4} &\equiv\mathbb{R}\mod1; \\
			L_{\leq 3}+L_{\leq 3}+L_{\leq 3} &\equiv \mathbb{R}\mod1; \\
			L_{\leq 3}+L_{\leq 3} &\not\equiv\mathbb{R}\mod1.
		\end{align*}
		Both $L_{\leq 3}+L_{\leq 3}$ and $L_{\leq 3}+L_{\leq 4}$ are not intervals.
	\end{theorem}
	% However, we are unable to determine whether $L_{\leq 3}+L_{\leq 4}+\mathbb{Z}$ equals $\mathbb{R}$. Based on observation, $L_{\leq 3}+L_{\leq 4}$ could contain the interval $[1,2]$, we make the following conjecture.
	% \ywl{
		However, we are unable to determine whether $L_{\leq 3}+L_{\leq 4}+\mathbb{Z}$ equals $\mathbb{R}$. Figure~\ref{fig: L_3 L_4} indicates that many sums in \(L_{\leq 3}+L_{\leq 4}\) cluster densely enough to suggest the presence of a genuine interval, say \([1,2]\). This observation naturally leads to the following conjecture.
		% }
	% \mbg{We are unable to determine whether $L_{\geq 3} + L_{\geq 4}+\mathbb{Z}\equiv \mathbb{R}$, but the behavior illustrated in Figure \ref{fig: L_3 L_4} does not rule out the possiblity that $L_{\geq 3}+L_{\geq 4}$ could contain the interval $[1,2].$}\mbg{Observing Figure ... it is entirely possible that $L_{\leq 3}+L_{\leq 4}$ contains the interval $[1,2]$, but using our methods we are not able to determine whether or not this is the case.}
	\begin{conjecture}
		\begin{align*}
			L_{\leq 3}+L_{\leq 4} \equiv \mathbb{R} \mod 1.
		\end{align*}
	\end{conjecture}
	
	By applying a result of Astels~\cite[Theorem 2.4]{astels2000} (see Proposition~\ref{prop: astels} in Section~\ref{section: GCS}), we give Theorem~\ref{thm: 2} which corresponds to the continued fraction setting as in~\cite[Theorem 1.5]{astels2000}.
	\begin{theorem}\label{thm: 2}
		Let $n\in\mathbb{N}$ and $k_1,k_2,\dots,k_n\in\mathbb{N}\setminus\{1\}$. Suppose
		\begin{align}
			\label{eq: (3)}
			\sum_{j=1}^{n}\frac{k_j-1}{k_j(k_j-1)-1}\geq1,
		\end{align}
		and, for any $j\in\{2,\ldots,n\}$,
		\begin{align}
			\label{eq: (4)}
			k_{j-1}\leq k_{j}\leq {{k_{j-1}}^2+2k_{j-1}+2}.
		\end{align}
		Then 
		\begin{align*}
			L_{\geq k_1}+L_{\geq k_2}+\cdots+L_{\geq k_n}\equiv\mathbb{R}\mod1.
		\end{align*}
	\end{theorem}
	Theorem~\ref{thm: 2} gives, for example,
	\begin{align*}
		L_{\geq3}+L_{\geq3}+L_{\geq5}\equiv\mathbb{R}\mod1, && 
		L_{\geq3}+L_{\geq4}+L_{\geq5}+L_{\geq6}\equiv\mathbb{R}\mod1.
	\end{align*}
	% Condition~\eqref{eq: (3)} guarantees that the total length of the L\"uroth sets is sufficiently large. Condition~\eqref{eq: (4)} guarantees that $(k_j)_j$ is not too spread out, so that the maximal gap of each L\"uroth set can be filled by the next one. Note that condition \eqref{eq: (3)} does not imply condition \eqref{eq: (4)}. For example, Theorem~\ref{thm: 2} says nothing to the following sumset: \begin{align*} L_{\geq 3}+L_{\geq 4}+L_{\geq 5}+L_{\geq 9}+L_{\geq 245}. \end{align*}
	Condition~\eqref{eq: (3)} guarantees that the total length of the L\"uroth sets involved is at least \(1\). More precisely, the sum in condition~\eqref{eq: (3)} is exactly the diameter of the sumset. In particular, if this sum is strictly less than \(1\), then
	\begin{align*}
		L_{\geq k_1} + L_{\geq k_2} + \cdots + L_{\geq k_n} \not\equiv \mathbb{R} \mod 1,
	\end{align*}
	simply because the total length is insufficient. The only case in which the sum in
	condition~\eqref{eq: (3)} is equal to \(1\) is \(n=1\) and \(k_1=2\), where the conclusion of Theorem~\ref{thm: 2} is trivially true.
	
	Condition~\eqref{eq: (4)} plays a different role: it ensures that the sequence \((k_j)_j\) is not too widely spread, so that the maximal gap of each L\"uroth set can be filled by the next one in the sum. Note that condition~\eqref{eq: (3)} does not imply condition~\eqref{eq: (4)}. For example, Theorem~\ref{thm: 2} does not apply to the sumset
	\begin{align*}
		L_{\geq 3} + L_{\geq 4} + L_{\geq 5} + L_{\geq 9} + L_{\geq 245}.
	\end{align*}
	It remains unknown whether a condition of the form~\eqref{eq: (4)} is actually necessary in Theorem~\ref{thm: 2} when the sum in condition~\eqref{eq: (3)} is greater than \(1\).
	
	Corollary~\ref{coro: 3} below is a direct application of Theorem~\ref{thm: 2}, and can be seen as an analogue to the result \eqref{eq: ksk} of Cusick--Lee. We do, however, give a separate proof of this corollary as well. 
	\begin{corollary}\label{coro: 3}
		For any $k\in\mathbb{N}\setminus\{1\}$, 
		\begin{align}\label{eq: Luroth Cusick--Lee Type}
			k L_{\geq k} \equiv \mathbb{R} \mod 1.
		\end{align}
	\end{corollary}
	While \eqref{eq: Luroth Cusick--Lee Type} can be compared with \eqref{eq: ksk}, it is worth noting that in the definition of $L_{\geq k}$, we do not include L\"uroth representations of finite lengths. We note that for $k=2$, \eqref{eq: Luroth Cusick--Lee Type} is trivially true, as $L_{\geq 2}=(0,1] \equiv \mathbb{R} \mod 1$. \eqref{eq: Luroth Cusick--Lee Type} is optimal in the sense that for $k\in\mathbb{N}$, if $k\geq3$ then
	\begin{align*}
		(k-1) L_{\geq k} \not\equiv \mathbb{R} \mod 1,
	\end{align*}
	as the diameter of $L_{\geq k}$ is not sufficient:
	\begin{align*}
		(k-1)\operatorname{diam}{L_{\geq k}}=\frac{(k-1)^2}{(k-1)k-1}=1-\frac{k-2}{k^2-k-1}<1.
	\end{align*}
	
	We also suggest a new type of sumsets to study: mixing L\"uroth sets whose digits are both bounded above and below, as in Theorem~\ref{thm: 4}.
	\begin{theorem}\label{thm: 4}
		For any $k\in\mathbb{N}\setminus\{1\}$, 
		\begin{align*}
			L_{\leq  k+2}+L_{\geq k}\equiv\mathbb{R}\mod 1.
		\end{align*}
	\end{theorem}
	One can observe that $\operatorname{diam}{(L_{\leq k}+L_{\geq k})}=1$ and suggest the following conjecture as a stronger statement than Theorem~\ref{thm: 4}.
	\begin{conjecture}\label{conj: 2}
		For any $k\in\mathbb{N}\setminus\{1\}$, 
		\begin{align*}
			L_{\leq k}+L_{\geq k}\equiv\mathbb{R}\mod 1.
		\end{align*}
	\end{conjecture}
	% As Theorems~\ref{thm: 1},~\ref{thm: 2} and Corollary~\ref{coro: 3} have corresponding results in the setting of continued fraction, one can ask if something can be said on the following sum: for any $k\in\mathbb{N}$,
	% \begin{align*}
		%     F_{\leq k}+F_{\geq k}.
		% \end{align*}
	% However, the diameter of the sum is not sufficient to make $F_{\leq k}+F_{\geq k}$ contain an interval of length at least 1, as for any $k\in\mathbb{N}$, if $k\geq2$ then
	% \begin{align*}
		%     \operatorname{diam}{\left(F_{\leq k}+F_{\geq k}\right)}
		%     =\frac{\sqrt{k^{2}+4k}-k}{2}-\frac{\sqrt{k^{2}+4k}-k}{2k}+\frac{1}{k}<1.
		% \end{align*}
	Theorems~\ref{thm: 1} and~\ref{thm: 2}, together with Corollary~\ref{coro: 3}, are motivated by their analogues in the continued–fraction setting. This naturally raises the question of whether the converse direction also holds; namely, whether Theorem~\ref{thm: 4} and Conjecture~\ref{conj: 2} admit counterparts for continued fractions. One may notice that for any $k\in\mathbb{N}$, if $k\geq2$ then
	\begin{align*}
		\operatorname{diam}{\left(F_{\leq k}+F_{\geq k}\right)}
		=\frac{\sqrt{k^{2}+4k}-k}{2}-\frac{\sqrt{k^{2}+4k}-k}{2k}+\frac{1}{k}<1;
	\end{align*}
	in particular, $F_{\leq k}+F_{\geq k}\not\equiv\mathbb{R}\mod1$, as the diameter of the sum is not sufficient to make it contain an interval of length at least 1. One can formulate the following alternative conjecture instead.
	\begin{conjecture}\label{conj: CF upper and lower}
		For any $k\in\mathbb{N}$, 
		\begin{align*}
			F_{\leq2k}+F_{\geq k}\equiv\mathbb{R}\mod 1.
		\end{align*}
	\end{conjecture}
	
	In the context of Diophantine approximation, the result of Hall implies that every real number can be decomposed as a sum of two badly approximable numbers. Erd\H{o}s~\cite{erdHos1962representations} invented an algorithm based on binary representation and proved that every real number is a sum of two Liouville numbers. In an upcoming work, Gayfulin--Nesharim~\cite{gayfulin2025realnumbersumreal}, claims that in Shulga's algorithm, the partial quotients of both summands are unbounded for all irrational numbers. Both results of Erd\H{o}s and Gayfulin--Nesharim imply that every real number is a sum of two well approximable numbers.
	
	The next natural question would be whether every real number can be decomposed as a sum of a badly approximable number and a well approximable number. Conjecture \ref{conj: CF upper and lower} above is in the spirit of this question, but does not capture it exactly: while $F_{\leq k}$ is a proper subset of the badly approximable numbers, and there exist numbers in $F_{\geq k}$ that are not well approximable. Nevertheless, Conjecture \ref{conj: CF bad and well} below asks this as a bigger question.
	\begin{conjecture}\label{conj: CF bad and well}
		\begin{align*}
			\bigcup_{k\in\mathbb{N}}F_{\leq k}+[0,1]\setminus\bigcup_{k\in\mathbb{N}}F_{\leq k}\equiv\mathbb{R}\mod1.
		\end{align*}
	\end{conjecture}
	
	\subsection{Dimension}\label{section: dimension}
	We give some results on the Hausdorff dimension of the L\"uroth sets and their sums. Theorem~\ref{thm: dim A} gives the Hausdorff dimension of sums of L\"uroth sets in terms of the respective Hausdorff dimensions.
	\begin{theorem}\label{thm: dim A}
		For any finite $\mathcal{A},\mathcal{B}\subseteq\mathbb{N}\setminus\{1\}$,
		\begin{align*}
			\dim{(L_\mathcal{A}+L_\mathcal{B})}
			=\min{\left\{1,\dim{L_\mathcal{A}}+\dim{L_\mathcal{B}}\right\}}.
		\end{align*}
	\end{theorem}
	By an application of~\cite[Theorem II]{Moran_1946}, the Hausdorff dimension of $L_{\mathcal{A}}$ is the unique solution $s\in[0,1]$ to 
	\begin{align*}
		\sum_{d\in\mathcal{A}}
		\left(\frac{1}{d(d-1)}\right)^{s} = 1.
	\end{align*}
	We can approximate that $\dim L_{\leq 3}=0.600967\ldots>1/2$ and consequently $\dim(L_{\leq 3}+L_{\leq 3})=1$.
	
	Theorem~\ref{thm: dim B} is an analogy to the lower bound estimate to the result of Good~\cite[Theorem 2]{Good_1941}. 
	\begin{theorem}\label{thm: dim B}
		For any $k\in\mathbb{N}\setminus\{1\}$,
		\begin{align*}
			\dim{L_{\geq k}}>\frac{1}{2}+\frac{1}{2\log{\max{\{16,k\}}}}.
		\end{align*}
	\end{theorem}
	Notice that for any $k,N\in\mathbb{N}\setminus\{1\}$, if $k<N$ then we have $\dim{L_{\geq k}}\geq\dim{L_{k,N}}$. Hence, we can find a lower bound of $\dim{L_{\geq k}}$ by finding the supremum of $s_{k,N}>0$ such that
	\begin{align*}
		\sum_{d=k}^{N}\left(\frac{1}{d(d-1)}\right)^{s_{k,N}}>1.
	\end{align*}
	For some small values of $k\in\mathbb{N}\setminus\{1\}$, we have,
	\begin{align*}
		\dim{L_{\geq3}}>0.8209; && \dim{L_{\geq4}}>0.7740; && \dim{L_{\geq5}}>0.7500; \\
		\dim{L_{\geq6}}>0.7347; && \dim{L_{\geq7}}>0.7239; && \dim{L_{\geq8}}>0.7157;
	\end{align*}
	and so on. The improvement is not significant for larger values of $k$.
	
	Shulga’s algorithm does not directly extend to the setting of L\"uroth expansions, for example in the case of $k=3$ and 
	\begin{align*}
		x
		=\frac{12}{41}
		=[\dot7]+[\dot7]
		\in\left(L_{\geq3}+L_{\geq3}\right)\cap\left[0,\frac{1}{2}\right],
	\end{align*}
	where the dot indicates that, in the L\"uroth expansion, all subsequent digits following the initial digit are identical to the digit marked with a dot. Direct application of Shulga’s algorithm to this case does not produce two L\"uroth expansions with digits bounded below by 3. Instead, the algorithm yields the following L\"uroth representations:
	\begin{align*}
		c=[4,25,12,5,5,6,6,6,5,5,5,\ldots], &&
		b=[11,16,8,2,4,2,2,2,2,2,3,\ldots].
	\end{align*}
	However, we propose a result and a conjecture, as analogies to the result of Shulga.
	\begin{theorem}\label{thm: dim C}
		For any $k\in\mathbb{N}\setminus\{1\}$,
		\begin{align*}
			\dim{\left(L_{\geq k}+L_{\geq k}\right)=1}.
		\end{align*}
	\end{theorem}
	\begin{conjecture}
		For any $k\in\mathbb{N}\setminus\{1\}$, $L_{\geq k}+L_{\geq k}$ contains an interval of positive length in $\mathbb{R}$.
	\end{conjecture}

	\section{Results on General Cantor Sumsets}\label{section: GCS}
	
	In this section, we introduce the definition and some known results from the literature on general Cantor sets. They are useful in proving our main results. General Cantor sets are defined in a similar construction to the middle-thirds Cantor set. It starts from a closed and bounded interval in $\mathbb{R}$, and proceeds by removing open intervals at each level. The fractals of $F_{\leq k}$ and $F_{\geq k}$ are examples of general Cantor sets, so are $L_{\leq k}$, $L_{\geq k}$ and many other self-similar sets.
	
	We generalise the middle-thirds Cantor set $C_{1/3}$ into the general Cantor set $C_{\mathcal{A}}$ by the following construction. Let $I$ be the minimal closed interval that contains the targeted fractal set. Put $I^1\coloneqq I$, which is Level 0 in the construction. Let $O^1\subseteq I^1$ be an open interval. Then, at Level 1, the two closed intervals $I^2$ and $I^3$ are formed by removing $O^1$ from $I^1$:
	\begin{align*}
		I^1
		= I^2 \cup O^1 \cup I^3.
	\end{align*}
	To construct a general Cantor set, at each step $i\in\mathbb{N}$, we remove an open interval $O^i$ from the closed, non-empty interval $I^i$, retaining the pair of closed, non-empty intervals $I^{2i}$ and $I^{2i+1}$, and continue this process iteratively across each level.
	Figure~\ref{fig: GCS} is an illustration.
	\begin{figure}[H]
		\centering
		
		\begin{tikzpicture}
			\useasboundingbox (0,-0.45) rectangle (\mytikzwidth,0.45);
			\fill [mycolour1] (0,-0.1) rectangle (\mytikzwidth,0.1);
			\draw (0,0)--(\mytikzwidth,0);
			\node at (1/2*\mytikzwidth,.45) {$I^1$};
		\end{tikzpicture}
		
		\begin{tikzpicture}
			\useasboundingbox (0,-0.45) rectangle (\mytikzwidth,0.45);
			\fill [mycolour1] (0/3*\mytikzwidth,-0.1) rectangle (1/3*\mytikzwidth,0.1);
			\fill [mycolour1] (2/3*\mytikzwidth,-0.1) rectangle (3/3*\mytikzwidth,0.1);
			\draw (0,0)--(\mytikzwidth,0);
			\node at (1/6*\mytikzwidth,.45) {$I^2$};
			\node at (1/2*\mytikzwidth,.45) {$O^1$};
			\node at (5/6*\mytikzwidth,.45) {$I^3$};
		\end{tikzpicture}
		
		\begin{tikzpicture}
			\useasboundingbox (0,-0.45) rectangle (\mytikzwidth,0.45);
			\fill [mycolour1] (0/9*\mytikzwidth,-0.1) rectangle (1/9*\mytikzwidth,0.1);
			\fill [mycolour1] (2/9*\mytikzwidth,-0.1) rectangle (3/9*\mytikzwidth,0.1);
			\fill [mycolour1] (6/9*\mytikzwidth,-0.1) rectangle (7/9*\mytikzwidth,0.1);
			\fill [mycolour1] (8/9*\mytikzwidth,-0.1) rectangle (9/9*\mytikzwidth,0.1);
			\draw (0,0)--(\mytikzwidth,0);
			\node at ( 1/18*\mytikzwidth,.45) {$I^4$};
			\node at ( 3/18*\mytikzwidth,.45) {$O^2$};
			\node at ( 5/18*\mytikzwidth,.45) {$I^5$};
			\node at (13/18*\mytikzwidth,.45) {$I^6$};
			\node at (15/18*\mytikzwidth,.45) {$O^3$};
			\node at (17/18*\mytikzwidth,.45) {$I^7$};
		\end{tikzpicture}
		
		$\vdots$
		
		\caption{A Construction of a General Cantor Set}
		\label{fig: GCS}
	\end{figure}
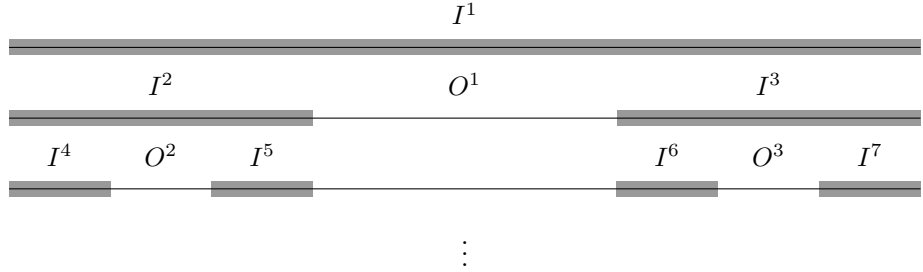
	The general Cantor set is obtained as $C\coloneqq I\setminus\bigcup_{i\in\mathbb{N}}O^i$, and $D\coloneqq ((I^i)_{i\in\mathbb{N}},(O^i)_{i\in\mathbb{N}})$ is referred to as a construction (or a derivation) of $C$. Note that there can be more than one construction that results in the same general Cantor set. For any interval $I\subseteq\mathbb{R}$, denote by $\lvert I \rvert$ the length of $I$. Hall~\cite{hall1947} proved the following result.
	
	\begin{proposition}[Hall {\cite[Theorem 2.2]{hall1947}}, 1947]
		Let $C_{\mathcal{A}}$ and $C_{\mathcal{B}}$ be two general Cantor sets with constructions $D_{\mathcal{A}}$ and $D_{\mathcal{B}}$. Write $I_{\mathcal{A}}^1=[x,x+a]$,  $I_{\mathcal{B}}^1=[y,y+b]$ and $e=\min\{a,b\}$. Suppose for any interval $I^{i}$ in either construction, with $I^{i} = I^{2i}\cup O^{i} \cup I^{2i+1}$, that
		\begin{align*}
			\lvert O^{i}\rvert
			\leq \min{\left\{\left\lvert I^{2i}\right\rvert, \left\lvert I^{2i+1}\right\rvert\right\}}.
		\end{align*}
		Then $[x+y, x+y+2e]\cup[x+y+a+b-2e, x+y+a+b]\subseteq C_{\mathcal{A}}+C_{\mathcal{B}}$. In particular, if $1/3\leq a/b\leq 3$, then $C_{\mathcal{A}}+C_{\mathcal{B}} = I_{\mathcal{A}}^{1}+I_{\mathcal{B}}^1$.
	\end{proposition}
	
	Two quantities are defined by Hlavka in~\cite{Hlavka1975} regarding a construction of a general Cantor set. Define $g_{\mathcal{A}}$ and $h_{\mathcal{A}}$ on the construction $D_{\mathcal{A}}$ of a general Cantor set.
	\begin{align*}
		g_{\mathcal{A}}
		=g(D_{\mathcal{A}})
		\coloneqq \sup_{j\in\mathbb{N}} \frac{\left\lvert O_{\mathcal{A}}^j \right\rvert}{\left\lvert I_{\mathcal{A}}^j \right\rvert}, &&
		h_{\mathcal{A}}
		=h(D_{\mathcal{A}})
		\coloneqq \inf_{j\in\mathbb{N}} \min{\left\{
			\frac{\left\lvert I_{\mathcal{A}}^{2j} \right\rvert}{\left\lvert I_{\mathcal{A}}^j \right\rvert},
			\frac{\left\lvert I_{\mathcal{A}}^{2j+1} \right\rvert}{\left\lvert I_{\mathcal{A}}^j \right\rvert}
			\right\}},
	\end{align*}
	where $g_{\mathcal{A}}$ captures the \emph{maximal removal ratio} and $h_{\mathcal{A}}$ captures the \emph{minimal retaining ratio} of the construction $D_{\mathcal{A}}$. By the quantities defined above, Hlavka proved the following results regarding the sum of two or more general Cantor sets.
	\begin{proposition}[Hlavka {\cite[Theorem 3]{Hlavka1975}}, 1975]\label{prop: hlavka3}
		Suppose there exist Cantor set constructions of $C_{\mathcal{A}}$ and $C_{\mathcal{B}}$ with all of the following conditions are satisfied:
		\begin{align*}
			g_{\mathcal{A}} g_{\mathcal{B}} \leq h_{\mathcal{A}} h_{\mathcal{B}}, &&
			g_{\mathcal{A}} \lvert I_{\mathcal{A}}\rvert \leq \lvert I_{\mathcal{B}} \rvert, &&
			g_{\mathcal{B}} \lvert I_{\mathcal{B}}\rvert \leq \lvert I_{\mathcal{A}} \rvert.
		\end{align*}
		Then $C_{\mathcal{A}}+C_{\mathcal{B}} = I_{\mathcal{A}}+I_{\mathcal{B}}$.
	\end{proposition}
	\begin{proposition}[Hlavka {\cite[Theorem 10]{Hlavka1975}}, 1975]\label{prop: hlavka10}
		Suppose there exist Cantor set constructions of $C_{\mathcal{A}_1}, \ldots, C_{\mathcal{A}_n}$ such that both of the following conditions are satisfied: for any $i,j \in \{1,\dots, n\}$,
		\begin{align*}
			h_{\mathcal{A}_i} \leq \frac{\lvert I_{\mathcal{A}_i} \rvert}{\lvert I_{\mathcal{A}_j} \rvert},
		\end{align*}
		and for any $i\in \{1,\dots, n\}$,
		\begin{align*}
			g_{\mathcal{A}_i}+h_{\mathcal{A}_i} \leq \sum_{j=1}^{n} h_{\mathcal{A}_j}.
		\end{align*}
		Then $C_{\mathcal{A}_1}+\dots+C_{\mathcal{A}_n} = I_{\mathcal{A}_1}+\dots+I_{\mathcal{A}_n}$.
	\end{proposition}
	The concept of \emph{thickness} was introduced by Astels in~\cite{astels2000}. For a construction $D_{\mathcal{A}}$ of the general Cantor set $C_{\mathcal{A}}$, we define
	\begin{align*}
		\tau(D_{\mathcal{A}})
		\coloneqq \inf_{i\in\mathbb{N}}{\left\{ \min \left\{
			\frac{\left\lvert I_{\mathcal{A}}^{2i}\right\rvert}{\left\lvert O_{\mathcal{A}}^{i}\right\rvert},
			\frac{\left\lvert I_{\mathcal{A}}^{2i+1}\right\rvert}{\left\lvert O_{\mathcal{A}}^{i}\right\rvert}\right\}\right\}}
		\geq\frac{h_{\mathcal{A}}}{g_{\mathcal{A}}}
		.
	\end{align*}
	Define the \emph{thickness} of the general Cantor set $C_{\mathcal{A}}$ by the supremum of $\tau(D_{\mathcal{A}})$ over all constructions $D_{\mathcal{A}}$ for $C_{\mathcal{A}}$, that is
	\begin{align*}
		\tau_{\mathcal{A}}
		=\tau(C_{\mathcal{A}})
		\coloneqq \sup_{D_{\mathcal{A}}}{\tau(D_{\mathcal{A}})},
	\end{align*}
	and also another quantity $\gamma$ on the general Cantor set $C_{\mathcal{A}}$:
	\begin{align*}
		\gamma_{\mathcal{A}}
		=\gamma(C_{\mathcal{A}})
		\coloneqq\frac{\tau_{\mathcal{A}}}{1+\tau_{\mathcal{A}}}\in[0,1).
	\end{align*}
	
	A construction $D=((I^i)_{i\in\mathbb{N}},(O^i)_{i\in\mathbb{N}})$ is said to be ordered, if for any $i\in\mathbb{N}$,
	\begin{align*}
		\left\lvert{O^i}\right\rvert
		\geq\max{\left\{\left\lvert{O^{2i}}\right\rvert,\left\lvert{O^{2i+1}}\right\rvert\right\}};
	\end{align*}
	that is, the lengths of removing intervals are always not greater than those at the previous level. One can observe that the following statements are equivalent to each other:
	\begin{itemize}
		\item $D=((I^i)_{i\in\mathbb{N}},(O^i)_{i\in\mathbb{N}})$ is ordered;
		\item for any $i\in\mathbb{N}$, $\lvert O^{i}\rvert \geq \sup\{\lvert O^{2i}\rvert, \lvert O^{2i+1}\rvert, \lvert O^{2(2i)}\rvert, \lvert O^{2(2i)+1}\rvert, \lvert O^{2(2i+1)}\rvert, \lvert O^{2(2i+1)+1}\rvert, \dots  \}$;
		\item for any $i,j\in\mathbb{N}$, if $I^i\subseteq I^j$ then $\lvert O^i\rvert\leq\lvert O^j\rvert$.
	\end{itemize}
	$D$ is said to be unordered if $D$ is not ordered. It is not hard to see that an ordered construction always exists. If $D_{\mathcal{A}}$ is ordered, then the thickness of the general Cantor set is exactly the thickness of $D_{\mathcal{A}}$; that is
	\begin{align}\label{thickness}
		\tau_{\mathcal{A}}=\tau(D_{\mathcal{A}})
		.
	\end{align}
	
	With the newly defined quantities, Astels established the following useful result concerning the sum of general Cantor sets. They are also applied in proving our main results. 
	\begin{proposition}[Astels {\cite[Theorem 2.4]{astels2000}}, 2000]\label{prop: astels}
		Let $S_{\gamma} = \gamma_{\mathcal{A}_1}+\cdots+\gamma_{\mathcal{A}_n}$. 
		\begin{enumerate}
			\item If $S_{\gamma} \geq 1$, then $C_{\mathcal{A}_1}+\dots+C_{\mathcal{A}_n}$ contains an interval. Otherwise, it contains a Cantor set of thickness at least 
			\begin{align*}
				\frac{S_{\gamma}}{1-S_{\gamma}}.
			\end{align*}
			Furthermore,
			\begin{align*}
				\dim{(C_{\mathcal{A}_1}+\cdots+C_{\mathcal{A}_n})}
				\geq \min{\left\{1,\frac{\log{2}}{\log\left(1+1/S_{\gamma}\right)}\right\}}
			\end{align*}
			\item Let $O_{\mathcal{A}_j}$ be a gap of maximal size in $C_{\mathcal{A}_j}$ (that is $\lvert O_{\mathcal{A}_j}\rvert=\sup_{i\in\mathbb{N}}{\lvert O_{\mathcal{A}_j}^i\rvert}$). Suppose for any $i\in\{2,\dots, n\}$ and $j\in\{1,\dots, i-1\}$,
			\begin{align}
				\label{eq: (12)}
				\lvert I_{\mathcal{A}_i} \rvert \geq \lvert O_{\mathcal{A}_j} \rvert,
			\end{align}
			and for any $i\in\{1,\dots, n-1\}$, 
			\begin{align}
				\label{eq: (13)}        
				\lvert I_{\mathcal{A}_1} \rvert+\cdots+\lvert I_{\mathcal{A}_i} \rvert
				\geq \lvert O_{\mathcal{A}_{i+1}} \rvert.
			\end{align}
			Then if $S_{\gamma} \geq 1$ we have $C_{\mathcal{A}_1}+\dots+C_{\mathcal{A}_n} = I_{\mathcal{A}_1}+\dots+I_{\mathcal{A}_n}$. Otherwise,
			\begin{align*}
				\tau_{C_{\mathcal{A}_1}+\dots+C_{\mathcal{A}_n}} \geq \frac{S_{\gamma}}{1-S_{\gamma}}.
			\end{align*}
		\end{enumerate}
	\end{proposition}
	
	\section{Constructions of L\"uroth Sets}\label{sect: Constructions}
	In this section, we give the constructions of the L\"uroth sets as general Cantor sets. The constructions below can be compared to the construction given in~\cite{lee2024explicitupperboundsdecay} to observe the differences.
	
	Let $N_1,N_2\in \mathbb{N}\setminus\{1\}$. Suppose $N_1<N_2$. Define $I_{N_1,N_2}\coloneqq [\langle N_2, N_1 \rangle]$, where for any $d\in\mathbb{N}\setminus\{1\}$, the left chevron and the right chevron are respectively defined by
	\begin{align*}
		\langle d
		&\coloneqq[d,\dot{N_2}]
		=\frac{1}{d}+\frac{1}{d(d-1)}\frac{N_2-1}{(N_2-1)N_2-1}, \\
		d\rangle
		&\coloneqq[d,\dot{N_1}]
		=\frac{1}{d}+\frac{1}{d(d-1)}\frac{N_1-1}{(N_1-1)N_1-1},
	\end{align*}
	where the dot indicates that, in the L\"uroth expansion, all subsequent digits following the initial digit are identical to the digit marked with a dot. When an expression includes at least one chevron, the order of operations is as follows: first, perform multiplication, division, and fractional operations; then, proceed with addition, chevron operations, and finally subtraction.
	
	We construct the L\"uroth set $L_{N_1,N_2}$ in the following \emph{Stepwise Complete Construction (SCC)} as a general Cantor set. At Level 0, we start with the following closed interval:
	%\begin{align*}
	%    I_{\mathcal{A}}^1 &= [[\dot{N_2}],[\dot{N_1}]].
	%\end{align*}
	\begin{align*}
		I_{N_1,N_2}^1
		\coloneqq I_{N_1,N_2}
		= \left[\langle N_2, N_1\rangle\right].
	\end{align*}
	For any $i\in\{1,\ldots,N_2-N_1\}$, at Level $i$, we have 
	\begin{align*}
		I_{N_1,N_2}^{2^i}
		= \left[\langle N_2,N_1+i \rangle\right], \,
		O_{N_1,N_2}^{2^{i-1}}
		= \left(N_1+i \rangle , \langle(N_1+i-1)\right), \,
		I_{N_1,N_2}^{2^i+1}
		= \left[\langle(N_1+i-1),(N_1+i-1)\rangle\right].
	\end{align*}
	The intervals $I_{N_1,N_2}^{2^i+1}$ and $I_{N_1,N_2}^{2^{(N_2-N_1)}}$ are all treated proportionally to $I_{\mathcal{A}}^{1}$ to give all further levels. We have $L_{N_1,N_2}=I_{N_1,N_2}\setminus\bigcup_{i\in\mathbb{N}}O_{N_1,N_2}^i$. The SCC for $L_{\leq 3}$, $L_{\leq 4}$, and $L_{\leq 5}$ illustrated in Figure~\ref{fig: L_3 L_4 L_5} as examples.
	
	To understand why the SCC yields $L_{N_1,N_2}$, observe that the numbers $\langle N_2$ and $N_1\rangle$ are respectively the smallest and the greatest elements in $L_{N_1,N_2}$. For each digit $d \in \{N_1, \dots, N_2\}$, the number $\langle d=[d, \dot{N_2}]$ is the smallest element in $L_{N_1,N_2}$ whose L\"uroth expansion begins with $d$, while $d\rangle=[d, \dot{N_1}]$ is the largest such element. Thus, the interval $[\langle d, d\rangle]$ contains all numbers whose L\"uroth expansions have first digit $d$, and all onward digits behave the same as those in $I_{N_1,N_2}$. Hence, the restriction on the first digit to $d$ is completed at each step (the last step, that is also at the last level, completes the restriction for digits first digit being $N_2-1$ and $N_2$ at the same time). The gap $(d+1\rangle, \langle d)$ consists of numbers whose second digit is less than $N_1$ or greater than $N_2$, and thus are excluded. This process is recursively applied to each retained subinterval through affine similarity, resulting in the same digit constraints on all further positions. Therefore, the SCC yields exactly the set $L_{N_1,N_2}$, realised as a general Cantor set.
	
	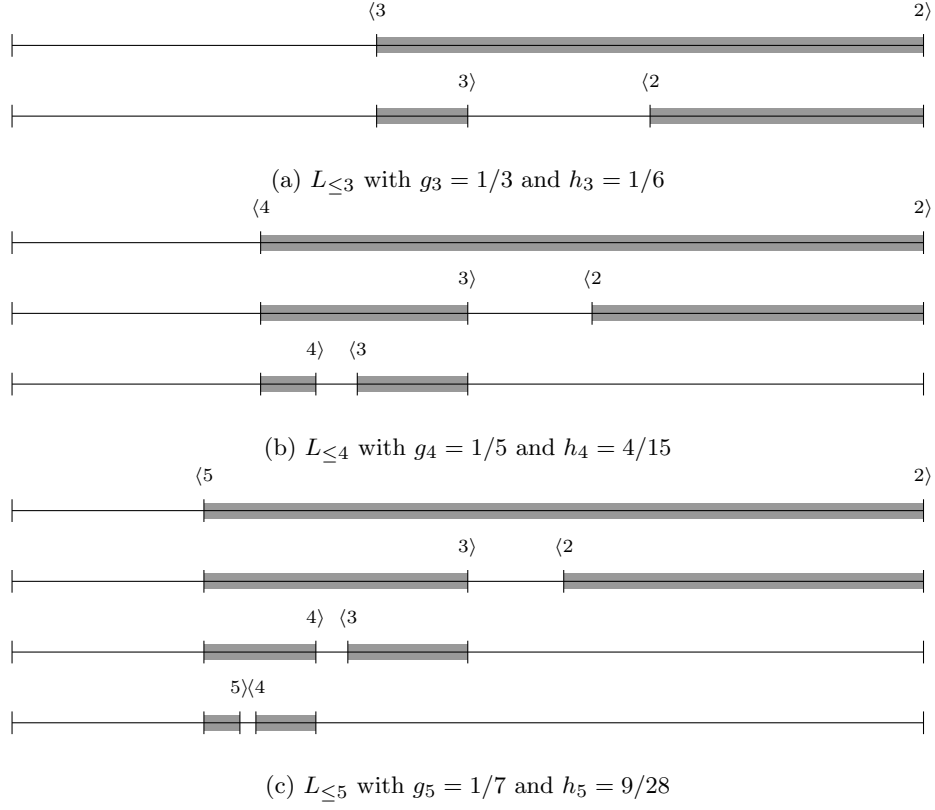
\begin{figure}
		\centering
		\begin{subfigure}[b]{\textwidth}
			\centering
			\scriptsize
			\begin{tikzpicture}
				\useasboundingbox (0,-0.45) rectangle (\mytikzwidth,0.45);
				\fill [mycolour1] ({h(3,3)},-0.1) rectangle ({h(2,2)},0.1);
				\draw (0,0)--(\mytikzwidth,0);
				\node at (1.0000\mytikzwidth,.45) {$2 \rangle$};
				\node at ({h(3,3)},.45) {$\langle 3$};
				\draw (0.0000000\mytikzwidth,-0.15)--+(0,0.3); 
				\draw (1.0000000\mytikzwidth,-0.15)--+(0,0.3);
				\draw ({h(3,3)},-0.15)--+(0,0.3);
			\end{tikzpicture}
			
			\begin{tikzpicture}
				\useasboundingbox (0,-0.45) rectangle (\mytikzwidth,0.45);
				\fill [mycolour1] ({h(3,3)},-0.1) rectangle ({h(3,2)},0.1);
				\fill [mycolour1] ({h(2,3)},-0.1) rectangle ({h(2,2)},0.1);
				\node at ({h(3,2)},.45) {$3 \rangle$};
				\node at ({h(2,3)},.45) {$\langle 2$};
				\draw (0,0)--(\mytikzwidth,0);
				\draw (1.0000000\mytikzwidth,-0.15)--+(0,0.3);
				\draw ({h(3,3)},-0.15)--+(0,0.3);
				\draw ({h(3,2)},-0.15)--+(0,0.3);
				\draw ({h(2,3)},-0.15)--+(0,0.3);
				\draw (0.0000000\mytikzwidth,-0.15)--+(0,0.3); 
			\end{tikzpicture}
			\caption{$L_{\leq 3}$ with $g_3=1/3$ and $h_3=1/6$}
			\label{fig: L_3}
		\end{subfigure}
		
		\begin{subfigure}[b]{\textwidth}
			\centering
			\scriptsize
			\begin{tikzpicture}
				\useasboundingbox (0,-0.45) rectangle (\mytikzwidth,0.45);
				\fill [mycolour1] ({h(4,4)},-0.1) rectangle ({h(2,2)},0.1);
				\draw (0,0)--(\mytikzwidth,0);
				\node at (1.0000\mytikzwidth,.45) {$2 \rangle$};
				\node at ({h(4,4)},.45) {$\langle 4$};
				\draw (0.0000000\mytikzwidth,-0.15)--+(0,0.3); 
				\draw (1.0000000\mytikzwidth,-0.15)--+(0,0.3);
				\draw ({h(4,4)},-0.15)--+(0,0.3);
			\end{tikzpicture}
			
			\begin{tikzpicture}
				\useasboundingbox (0,-0.45) rectangle (\mytikzwidth,0.45);
				\fill [mycolour1] ({h(4,4)},-0.1) rectangle ({h(3,2)},0.1);
				\fill [mycolour1] ({h(2,4)},-0.1) rectangle ({h(2,2)},0.1);
				\node at ({h(3,2)},.45) {$3 \rangle$};
				\node at ({h(2,4)},.45) {$\langle 2$};
				\draw (0,0)--(\mytikzwidth,0);
				\draw (1.0000000\mytikzwidth,-0.15)--+(0,0.3);
				\draw ({h(4,4)},-0.15)--+(0,0.3);
				\draw ({h(3,2)},-0.15)--+(0,0.3);
				\draw ({h(2,4)},-0.15)--+(0,0.3);
				\draw (0.0000000\mytikzwidth,-0.15)--+(0,0.3); 
			\end{tikzpicture}
			
			\begin{tikzpicture}
				\useasboundingbox (0,-0.45) rectangle (\mytikzwidth,0.45);
				\fill [mycolour1] ({h(4,4)},-0.1) rectangle ({h(4,2)},0.1);
				\fill [mycolour1] ({h(3,4)},-0.1) rectangle ({h(3,2)},0.1);
				\node at ({h(4,2)},.45) {$4 \rangle$};
				\node at ({h(3,4)},.45) {$\langle 3$};
				\draw (0,0)--(\mytikzwidth,0);
				\draw (1.0000000\mytikzwidth,-0.15)--+(0,0.3);
				\draw ({{h(4,4)}},-0.15)--+(0,0.3);
				\draw ({{h(4,2)}},-0.15)--+(0,0.3);
				\draw ({{h(3,4)}},-0.15)--+(0,0.3);
				\draw ({{h(3,2)}},-0.15)--+(0,0.3);
				\draw (0.0000000\mytikzwidth,-0.15)--+(0,0.3); 
			\end{tikzpicture}
			
			\caption{$L_{\leq 4}$ with $g_4=1/5$ and $h_4=4/15$}
			\label{fig: L_4}
		\end{subfigure}
		
		\begin{subfigure}[b]{\textwidth}
			\centering
			\scriptsize
			\begin{tikzpicture}
				\useasboundingbox (0,-0.45) rectangle (\mytikzwidth,0.45);
				\fill [mycolour1] ({h(5,5)},-0.1) rectangle ({h(2,2)},0.1);
				\draw (0,0)--(\mytikzwidth,0);
				\node at (1.0000\mytikzwidth,.45) {$2 \rangle$};
				\node at ({h(5,5)},.45) {$\langle 5$};
				\draw (0.0000000\mytikzwidth,-0.15)--+(0,0.3); 
				\draw (1.0000000\mytikzwidth,-0.15)--+(0,0.3);
				\draw ({h(5,5)},-0.15)--+(0,0.3);
			\end{tikzpicture}
			
			\begin{tikzpicture}
				\useasboundingbox (0,-0.45) rectangle (\mytikzwidth,0.45);
				\fill [mycolour1] ({h(5,5)},-0.1) rectangle ({h(3,2)},0.1);
				\fill [mycolour1] ({h(2,5)},-0.1) rectangle ({h(2,2)},0.1);
				\node at ({h(3,2)},.45) {$3 \rangle$};
				\node at ({h(2,5)},.45) {$\langle 2$};
				\draw (0,0)--(\mytikzwidth,0);
				\draw (1.0000000\mytikzwidth,-0.15)--+(0,0.3);
				\draw ({h(5,5)},-0.15)--+(0,0.3);
				\draw ({h(3,2)},-0.15)--+(0,0.3);
				\draw ({h(2,5)},-0.15)--+(0,0.3);
				\draw (0.0000000\mytikzwidth,-0.15)--+(0,0.3); 
			\end{tikzpicture}
			
			\begin{tikzpicture}
				\useasboundingbox (0,-0.45) rectangle (\mytikzwidth,0.45);
				\fill [mycolour1] ({h(5,5)},-0.1) rectangle ({h(4,2)},0.1);
				\fill [mycolour1] ({h(3,5)},-0.1) rectangle ({h(3,2)},0.1);
				\node at ({h(4,2)},.45) {$4 \rangle$};
				\node at ({h(3,5)},.45) {$\langle 3$};
				\draw (0,0)--(\mytikzwidth,0);
				\draw (1.0000000\mytikzwidth,-0.15)--+(0,0.3);
				\draw ({{h(5,5)}},-0.15)--+(0,0.3);
				\draw ({{h(4,2)}},-0.15)--+(0,0.3);
				\draw ({{h(3,5)}},-0.15)--+(0,0.3);
				\draw ({{h(3,2)}},-0.15)--+(0,0.3);
				\draw (0.0000000\mytikzwidth,-0.15)--+(0,0.3); 
			\end{tikzpicture}
			
			\begin{tikzpicture}
				\useasboundingbox (0,-0.45) rectangle (\mytikzwidth,0.45);
				\fill [mycolour1] ({h(5,5)},-0.1) rectangle ({h(5,2)},0.1);
				\fill [mycolour1] ({h(4,5)},-0.1) rectangle ({h(4,2)},0.1);
				\node at ({h(5,2)},.45) {$5 \rangle$};
				\node at ({h(4,5)},.45) {$\langle 4$};
				\draw (0,0)--(\mytikzwidth,0);
				\draw (1.0000000\mytikzwidth,-0.15)--+(0,0.3);
				\draw ({{h(5,5)}},-0.15)--+(0,0.3);
				\draw ({{h(5,2)}},-0.15)--+(0,0.3);
				\draw ({{h(4,5)}},-0.15)--+(0,0.3);
				\draw ({{h(4,2)}},-0.15)--+(0,0.3);
				\draw (0.0000000\mytikzwidth,-0.15)--+(0,0.3); 
			\end{tikzpicture}
			
			\label{fig: L_5}
			\caption{ $L_{\leq 5}$ with $g_5=1/7$ and $h_5=9/28$}
		\end{subfigure}
		
		\caption{Stepwise Complete Constructions (SCCs) for $L_{\leq 3}$, $L_{\leq 4}$, and $L_{\leq 5}$ in $[0,1]$}
		\label{fig: L_3 L_4 L_5}
	\end{figure}
	
	For example, in $L_{\leq 3}$, we have at Level 0:
	\begin{align*}
		I_3^1=I_3=[\langle3,2\rangle]=\left[\frac{2}{5},1\right];
	\end{align*}
	and at Level 1:
	\begin{align*}
		I_3^2=[\langle3,3\rangle]=\left[\frac{2}{5},\frac{1}{2}\right], &&
		O_3^1=(3\rangle,\langle2)=\left(\frac{1}{2},\frac{7}{10}\right), &&
		I_3^3=[\langle2,2\rangle]=\left[\frac{7}{10},1\right];
	\end{align*}
	$I_3^2$ and $I_3^3$ are both treated proportionally to $I_{3}^{1}$ to give all further levels.
	
	In general, for L\"uroth sets of the form $L_{N_1,N_2}$, the $g$- and $h$- values of the SCCs are given by their finite steps due to self-similarity. The following proposition gives the lengths of the intervals in the SCCs; we can also see that all SCCs are ordered.
	\begin{proposition}\label{prop: Stepwise Complete g h}
		Let $N_1,N_2\in\mathbb{N}\setminus\{1\}$. Suppose $N_1<N_2$. Under SCC, for any $i\in\{1,\ldots,N_2-N_1\}$, at Level $i$,
		\begin{itemize}
			\item the main interval of length (the interval on the left at Level $i-1$):
			\begin{align*}
				\left\lvert I_{N_1,N_2}^{2^{i-1}}\right\rvert
				=(N_1+i-1)\rangle-\langle N_2;
			\end{align*}
			\item the interval on the left of length:
			\begin{align*}
				\left\lvert I_{N_1,N_2}^{2^i}\right\rvert
				=N_1+i\rangle-\langle N_2;
			\end{align*}
			\item the gap of length (strictly decreasing):
			\begin{align}
				\label{eq: SCC gap length}
				\left\lvert O_{N_1,N_2}^{2^{i-1}}\right\rvert
				=\langle (N_1+i-1)-N_1+i\rangle;
			\end{align}
			\item the interval on the right of length:
			\begin{align*}
				\left\lvert I_{N_1,N_2}^{2^i+1}\right\rvert
				=(N_1+i-1)\rangle-\langle(N_1+i-1).
			\end{align*}
		\end{itemize}
	\end{proposition}
	
	\begin{proposition}
		Let $N_1,N_2\in\mathbb{N}\setminus\{1\}$. Suppose $N_1<N_2$. The SCC for $L_{N_1,N_2}$ is ordered.
	\end{proposition}
	\begin{proof}
		By self-similarity, it suffices to prove that $i\in\{1,\ldots,N_2-N_1-1\}$,
		\begin{align*}
			\left\lvert O_{N_1,N_2}^{2^{i-1}}\right\rvert
			\geq\max{\left\{\left\lvert O_{N_1,N_2}^{2^{i}}\right\rvert,\left\lvert O_{N_1,N_2}^{2^i+1}\right\rvert\right\}}.
		\end{align*}
		On the one hand, from \eqref{eq: SCC gap length}, one has that for any $i\in\{1,\ldots,N_2-N_1-1\}$,
		\begin{align*}
			\left\lvert O_{N_1,N_2}^{2^{i-1}}\right\rvert
			&= \frac{1}{N_1+i-1}+\frac{\langle N_2}{(N_1+i-1)(N_1+i-2)}-\frac{1}{N_1+i}-\frac{N_1\rangle}{(N_1+i)(N_1+i-1)}.
		\end{align*}
		One can see the following difference is positive:
		\begin{align*}
			\Delta_{N_1,N_2,i,0}
			&\coloneqq \left\lvert O_{N_1,N_2}^{2^{i-1}}\right\rvert-\left\lvert O_{N_1,N_2}^{2^{i}}\right\rvert \\
			&=\frac{1}{N_1+i-1}-\frac{2}{N_1+i}+\frac{1}{N_1+i+1}+\frac{\langle N_2}{N_1+i-1}\left(\frac{1}{N_1+i-2}-\frac{1}{N_1+i}\right) \\
			&\qquad-\frac{N_1 \rangle}{N_1+i}\left(\frac{1}{N_1+i-1}-\frac{1}{N_1+i+1}\right) \\
			&>\frac{1}{N_1+i-1}-\frac{2}{N_1+i}+\frac{1}{N_1+i+1}-\frac{N_1 \rangle}{N_1+i}\left(\frac{1}{N_1+i-1}-\frac{1}{N_1+i+1}\right) \\
			&\geq N_1 \rangle\left(\frac{1}{N_1+i-1}-\frac{2}{N_1+i}+\frac{1}{N_1+i+1}\right)-\frac{N_1 \rangle}{N_1+i}\left(\frac{1}{N_1+i-1}-\frac{1}{N_1+i+1}\right) \\
			&=N_1 \rangle\left(\frac{1}{N_1+i-1}-\frac{2}{N_1+i}+\frac{1}{N_1+i+1}-\frac{1}{N_1+i}\left(\frac{1}{N_1+i-1}-\frac{1}{N_1+i+1}\right)\right) \\
			&=0.
		\end{align*}
		Hence, \eqref{eq: SCC gap length} is strictly decreasing, that is
		for any $i\in\{1,\ldots,N_2-N_1-1\}$,
		\begin{align*}
			\left\lvert O_{N_1,N_2}^{2^{i-1}}\right\rvert
			\geq \left\lvert O_{N_1,N_2}^{2^{i}}\right\rvert.
		\end{align*}
		
		On the other hand, from similarity, one has that for any $i\in\{1,\ldots,N_2-N_1-1\}$,
		\begin{align*}
			\left\lvert O_{N_1,N_2}^{2^i+1}\right\rvert
			=\frac{\left\lvert O_{N_1,N_2}^1\right\rvert}{\left\lvert I_{N_1,N_2}^1\right\rvert}\left\lvert I_{N_1,N_2}^{2^i+1}\right\rvert
			=\alpha_{N_1,N_2}\left((N_1+i-1)\rangle-\langle (N_1+i-1)\right)
			,
		\end{align*}
		where
		\begin{align*}
			\alpha_{N_1,N_2}
			\coloneqq\frac{\left\lvert O_{N_1,N_2}^1\right\rvert}{\left\lvert I_{N_1,N_2}^1\right\rvert}
			=\frac{\langle N_1-N_1+1\rangle}{N_1\rangle-\langle N_2}
			.
		\end{align*}
		One can see that the following difference is positive:
		\begin{align*}
			\Delta_{N_1,N_2,i,1}
			&\coloneqq \left\lvert O_{N_1,N_2}^{2^{i-1}}\right\rvert-\left\lvert O_{N_1,N_2}^{2^i+1}\right\rvert \\
			&=\left(1+\alpha_{N_1,N_2}\right)\langle(N_1+i-1)-N_1+i\rangle-(N_1+i-1)\rangle\alpha_{N_1,N_2} \\
			&=\frac{1}{N_1+i-1}\left(1+\frac{\langle N_2-\alpha_{N_1,N_2}(N_1\rangle-\langle N_2)}{N_1+i-2}-\frac{N_1\rangle}{N_1+i}\right)-\frac{1}{N_1+i} \\
			&=\frac{1}{N_1+i-1}\left(1+\frac{N_1+1\rangle-\langle N_1+\langle N_2}{N_1+i-2}-\frac{N_1\rangle}{N_1+i}\right)-\frac{1}{N_1+i} \\
			&\geq\frac{1}{N_1+i-1}\left(1+\frac{N_1+1\rangle-1/N_1}{N_1+i-2}-\frac{N_1\rangle}{N_1+i}\right)-\frac{1}{N_1+i} \\
			&=\frac{1}{N_1+i-1}\left(1+\frac{1}{N_1+i-2}\frac{2-N_1}{{N_1}^{3}-2N_1-1}-\frac{1}{N_1+i}\frac{N_1-1}{N_1\left(N_1-1\right)-1}\right)-\frac{1}{N_1+i} \\
			&=\frac{(N_1-2)\left({N_1}^3+(i-1){N_1}^2+(i-3)N_1-i\right)}{(N_1+i-2)(N_1+i-1)(N_1+i)(N_1+1)({N_1}^2-N_1-1)} \\
			&\geq0
			.
		\end{align*}
		Hence, for any $i\in\{1,\ldots,N_2-N_1-1\}$,
		\begin{align*}
			\left\lvert O_{N_1,N_2}^{2^{i-1}}\right\rvert
			\geq\left\lvert O_{N_1,N_2}^{2^i+1}\right\rvert.
		\end{align*}
	\end{proof}
	Since SCC are ordered, Proposition~\ref{prop: Stepwise Complete g h} also helps to compute the thickness of L\"uroth sets of the form $L_{N_1,N_2}$ and their $\gamma$-values. This is discussed in Section~\ref{section: thickness}.
	
	\section{Thickness of L\"uroth Sets}\label{section: thickness}
	In this section, we establish a closed expression of thickness $\tau_{N_1,N_2}$ for L\"uroth sets of the form $L_{N_1,N_2}$. The construction SCC from the previous section is ordered, so by \eqref{thickness} it suffices to find the thickness of the SCC for $L_{N_1,N_2}$. The following proposition serves as a foundation for most of the main results in this paper:
	\begin{proposition}\label{prop: Stepwise Complete tau}
		Let $N_1,N_2\in\mathbb{N}\setminus\{1\}$. Suppose $N_1<N_2$. Then the thickness of $L_{N_1,N_2}$ is
		\begin{align}
			\label{eq: tau L N1 N2}
			\tau_{N_1,N_2}
			&=\frac{N_2\rangle-\langle N_2}{\langle(N_2-1)-N_2\rangle} \\
			&=\frac{(N_2-N_1)(N_2-2)(N_1N_2-N_1-N_2+2)}{{N_1}^2({N_2}^3-2{N_2}^2+2)-{N_1}(2{N_2}^3-5{N_2}^2+{N_2}+4)-({N_2}^2-{N_2})} \nonumber
			.
		\end{align}
	\end{proposition}
	In other words, the supremum $\tau_{N_1,N_2}$ is realised at the left interval at the last step in the SCC. By the definition of SCC and since it is ordered, we have:
	\begin{align*}
		\tau_{N_1,N_2}
		&=\min_{x\in\{0,1,\ldots,N_2-N_1-1\}}{\left\{\min{\left\{
				\frac{N_1+x+1\rangle-\langle N_2}{\langle N_1+x-N_1+x+1\rangle}
				,
				\frac{N_1+x\rangle-\langle N_1+x}{\langle N_1+x-N_1+x+1\rangle}
				\right\}}\right\}}.
	\end{align*}
	Define $f_0,f_1:[0,+\infty)\to\mathbb{R}$ by, for any $x\geq0$, 
	\begin{align*}
		&f_0(x)=f_{N_1,N_2,0}(x)
		\coloneqq \frac{N_1+x+1\rangle-\langle N_2}{\langle N_1+x-N_1+x+1\rangle}, \\
		&f_1(x)=f_{N_1,N_2,1}(x)
		\coloneqq \frac{N_1+x\rangle-\langle N_1+x}{\langle N_1+x-N_1+x+1\rangle}.
	\end{align*}
	To prove Proposition~\ref{prop: Stepwise Complete tau}, it suffices to prove the following lemmas as together they conclude that
	\begin{align*}
		\tau_{N_1,N_2}=f_0(N_2-N_1-1).
	\end{align*}
	The first lemma takes care of the thickness of $L_{\leq 3}$, $L_{\leq 4}$ and $L_{\leq 5}$.
	\begin{lemma}\label{lemma: tau A}
		For $N_1=2$ and any $N_2\in\{3,4,5\}$, \eqref{eq: tau L N1 N2} holds.
	\end{lemma}
	The next four lemmas combine to take care of $L_{\leq k}$ for $k\geq 6$ and $L_{N_1,N_2}$ for $N_1\geq 3$.
	\begin{lemma}\label{lemma: tau B}
		For $N_1=2$ and any $N_2\geq 6$, 
		\begin{align*}
			f_0(0)>f_0(N_2-3),
		\end{align*}
		and for any $x\in[1,N_2-4]$,
		\begin{align*}
			f_0(x)\geq f_1(x).
		\end{align*}
	\end{lemma}
	\begin{lemma}\label{lemma: tau C}
		For any $N_1\geq3$, $N_2\geq N_1+2$ and $x\in[0,N_2-N_1-2]$,
		\begin{align*}
			f_0(x)\geq f_1(x).
		\end{align*}
	\end{lemma}
	\begin{lemma}\label{lemma: tau D}
		For any $N_1,N_2\in\mathbb{N}\setminus\{1\}$, if $N_1<N_2$ then $f_1$ is strictly decreasing.
	\end{lemma}
	\begin{lemma}\label{lemma: tau E}
		For any $N_1,N_2\in\mathbb{N}\setminus\{1\}$, if $N_1<N_2$ then
		\begin{align*}
			f_0(N_2-N_1-1)<f_1(N_2-N_1-1).
		\end{align*}
	\end{lemma}
	Lemmas~\ref{lemma: tau B}, \ref{lemma: tau D}, \ref{lemma: tau E} work together to find the thickness for $L_{\leq k}$ for $k\geq 6$; Lemmas~\ref{lemma: tau C}, \ref{lemma: tau D}, \ref{lemma: tau E} work together to find the thickness of $L_{N_1,N_2}$ for $N_1\geq 3$. All of that combined proves Proposition~\ref{prop: Stepwise Complete tau}.
	
	\begin{proof}[Proof of Lemma~\ref{lemma: tau A}]
		We calculate the exact values of $\tau_3$, $\tau_4$ and $\tau_5$. Based on the SCC for $L_{\leq 3}$, $L_{\leq 4}$ and $L_{\leq 5}$, we have
		\begin{align*}
			\tau_3
			&= \min\left\{ \frac{3\rangle-\langle3}{\langle2 -3\rangle},\frac{2\rangle-\langle 2}{\langle2-3\rangle}\right\}
			&&=\frac{1}{7}=\frac{3\rangle-\langle 3}{\langle 2-3\rangle}; \\
			\tau_4
			&= \min\left\{\frac{3\rangle-\langle4}{\langle 2-3\rangle}, \frac{2\rangle-\langle2}{\langle 2-3\rangle},\frac{4\rangle-\langle 4}{\langle 3-4\rangle},\frac{3\rangle-\langle 3}{\langle3-4\rangle}\right\}
			&&=\frac{4}{7}=\frac{4\rangle-\langle 4}{\langle 3-4\rangle}; \\
			\tau_5
			&= \min\left\{\frac{3\rangle-\langle5}{\langle 2-3\rangle},\frac{2\rangle-\langle2}{\langle 2-3\rangle}, \frac{4\rangle-\langle 5}{\langle3-4\rangle},\frac{3\rangle-\langle 3}{\langle3-4\rangle}, \frac{5\rangle-\langle 5}{\langle4-5\rangle},\frac{4\rangle-\langle 4}{\langle4-5\rangle} \right\} 
			&&=\frac{9}{13}=\frac{5\rangle-\langle5}{\langle 4- 5\rangle}.
		\end{align*}
		The result follows.
	\end{proof}
	
	In Lemmas \ref{lemma: tau B}, \ref{lemma: tau C} and \ref{lemma: tau E}, we are comparing $f_0$ and $f_1$, so let us look at the numerators.
	We have, for any $x\geq0$,
	\begin{align*}
		N_1+x+1\rangle-\langle N_2
		&=\frac{1}{N_1+x+1}+\frac{1}{(N_1+x+1)(N_1+x)}N_1\rangle-\langle N_2 \\
		&=\frac{1}{N_1+x+1}\left(1+\frac{N_1\rangle}{N_1+x}\right)-\langle N_2
		;
	\end{align*}
	and
	\begin{align*}
		N_1+&x\rangle -\langle N_1+x \\
		&=\frac{1}{N_1+x}+\frac{1}{(N_1+x)(N_1+x-1)}N_1\rangle-\frac{1}{N_1+x}-\frac{1}{(N_1+x)(N_1+x-1)}\langle N_2 \\
		&=\frac{N_1\rangle}{(N_1+x)(N_1+x-1)}-\frac{\langle N_2}{(N_1+x)(N_1+x-1)}\\
		&=\frac{N_1\rangle-\langle N_2}{(N_1+x)(N_1+x-1)}.
	\end{align*}
	Define $D:[0,+\infty)\to\mathbb{R}$ to be the difference of the numerators of $f_0$ and $f_1$, that is for any $x\geq0$,
	\begin{align}
		\label{eq: def D}
		D(x)
		&\coloneqq \frac{1}{N_1+x+1}\left(1+\frac{N_1\rangle}{N_1+x}\right)-\langle N_2-\frac{N_1\rangle-\langle N_2}{(N_1+x)(N_1+x-1)}.
	\end{align}
	\begin{proof}[Proof of Lemma~\ref{lemma: tau B}]
		Suppose $N_1=2$ and $N_2\geq6$. To prove $f_0(0)-f_0(N_2-3) > 0$, it suffices to show that
		\begin{align*}
			(3\rangle-\langle N_2)(\langle (N_2-1)-N_2\rangle)- (N_2\rangle-\langle N_2)(\langle 2-3\rangle) > 0
		\end{align*}
		Writing out the full expression of the left-hand side above yields
		\begin{align*}
			&\left( \frac13 +\frac16-\frac{N_2-1}{{N_2}^2-N_2-1} \right) \left( \frac{1}{N_2-1}+\frac{1}{(N_2-1)(N_2-2)}\frac{N_2-1}{{N_2}^2-N_2-1}-\frac{1}{N_2}-\frac{1}{N_2(N_2-1)}\right) \\
			&\qquad-\left(  \frac{1}{N_2}+\frac{1}{N_2(N_2-1)}-\frac{N_2-1}{{N_2}^2-N_2-1} \right) \left( \frac12+\frac12 \frac{N_2-1}{{N_2}^2-N_2-1} -\frac13-\frac16 \right) \\
			&= \frac{{N_2}^2-3N_2+1}{2({N_2}^2-N_2-1)}\frac{1}{(N_2-2)({N_2}^2-N_2-1)}-\frac{N_2-2}{(N_2-1)({N_2}^2-N_2-1)}\frac{N_2-1}{2({N_2}^2-N_2-1)} \\
			&=\frac{N_2-3}{2(N_2-2)({N_2}^2-N_2-1)^2}>0,
		\end{align*}
		which is positive since we assumed $N_2\geq 6$.
		
		For the second claim we show that for any $1\leq x\leq N_2-4$, $D(x)\geq 0$. To see this, we have 
		\begin{align*}
			D(x)
			&=\frac{1}{N_1+x+1}\left(1+\frac{N_1\rangle}{N_1+x}\right)-\langle N_2-\frac{N_1\rangle-\langle N_2}{(N_1+x)(N_1+x-1)} \\
			&=\frac{1}{2+x}-\langle N_2-\frac{1-\langle N_2}{(2+x)(1+x)} \\
			&=\frac{ -x^2\langle N_2 +x(1-3\langle N_2)  \langle N_2}{(1+x)(2+x)}
			,
		\end{align*}
		where the numerator is non-negative if and only if
		\begin{align}
			\label{eq: D positive equivalent}
			\frac{(1-3\langle N_2)-\sqrt{(1-3\langle N_2)^2-4(\langle N_2)^2}}{2\langle N_2}
			\leq x \leq
			\frac{(1-3\langle N_2)+\sqrt{(1-3\langle N_2)^2-4(\langle N_2)^2}}{2\langle N_2}.
		\end{align}
		Notice that for any $N_2 \geq 6$,
		\begin{align*}
			\Delta_{N_2,0}
			&\coloneqq 1-\frac{(1-3\langle N_2)-\sqrt{(1-3\langle N_2)^2-4(\langle N_2)^2}}{2\langle N_2} \\
			&=\frac{-4+6{N_2}-{N_2}^{2}+\sqrt{\left(-4+6{N_2}-{N_2}^{2}\right)^2+4{N_2}^3-28{N_2}^2+40{N_2}-16}}{2({N_2}-1)}
			\geq0,
		\end{align*}
		and
		\begin{align*}
			\Delta_{N_2,1}
			&\coloneqq \frac{(1-3\langle N_2)+\sqrt{(1-3\langle N_2)^2-4(\langle N_2)^2}}{2\langle N_2}-(N_2-4) \\
			&=\frac{-6+6{N_2}-{N_2}^{2}+\sqrt{\left(-6+6{N_2}-{N_2}^{2}\right)^2+4{N_2}^3-32{N_2}^2+64{N_2}-36}}{2({N_2}-1)}
			\geq0.
		\end{align*}
		Thus, \eqref{eq: D positive equivalent} holds for all $1\leq x\leq N_2-4$, which also completes the proof.
	\end{proof}
	
	Now we handle the cases when $N_1$ is large.
	\begin{proof}[Proof of Lemma~\ref{lemma: tau C}]
		Suppose $N_1\geq 3$ and $N_2-N_1-2\geq0$. For any $x\geq0$,
		\begin{align*}
			D(x)=\frac{c_0+c_1x+c_2x^2+c_3x^3}{C_{N_1,N_2}(N_1+x-1)(N_1+x)(N_1+x+1)},
		\end{align*}
		where the coefficients $c_0,c_1,c_2,c_3$ and the constant $C_{N_1,N_2}>0$ are in terms of $N_1$ and $N_2$,
		\begin{align*}
			c_0&=-{N_1}^5{N_2}+{N_1}^5+{N_1}^4 {N_2}^2-2 {N_1}^4-2 {N_1}^3 {N_2}^2+5 {N_1}^3 {N_2}-{N_1}^3-{N_1}^2 {N_2}+{N_1}^2-{N_1} {N_2}^2\\
			&\qquad-2 {N_1} {N_2}+4 {N_1}+2 {N_2}^2-3 {N_2}-1, \\
			c_1&=-3 {N_1}^4 {N_2}+3 {N_1}^4+2 {N_1}^3 {N_2}^2+{N_1}^3 {N_2}-5 {N_1}^3-3 {N_1}^2 {N_2}^2+8 {N_1}^2 {N_2}-2 {N_1}^2-{N_1} {N_2}^2 \\
			&\qquad-{N_1} {N_2}+3 {N_1}+{N_2}^2-3 {N_2}+1, \\
			c_2&=-3 {N_1}^3 {N_2}+3 {N_1}^3+{N_1}^2 {N_2}^2+2 {N_1}^2 {N_2}-4 {N_1}^2-{N_1} {N_2}^2+4 {N_1} {N_2}-2 {N_1}-{N_2}^2+{N_2}+1, \\
			c_3
			&=-{N_1}^2{N_2}+{N_1}^2+{N_1} {N_2}-{N_1}+{N_2}-1 \\
			&=-({N_1}^2-N_1-1)(N_2-1)<0, \\
			C_{N_1,N_2}
			&= ({N_1}^2-N_1-1)({N_2}^2-N_2-1).
		\end{align*}
		Define $E:\mathbb{R}\to\mathbb{R}$ to be the cubic equation in the numerator of $D$, that is for any $x\in\mathbb{R}$,
		\begin{align*}
			E(x)\coloneqq{c_0+c_1x+c_2x^2+c_3x^3}.
		\end{align*}
		It now suffices to prove the following properties for $E$:
		\begin{enumerate}[A.]
			\item $c_3<0$; hence $\lim_{x\to-\infty} E(x) =+\infty$ and $\lim_{x\to+\infty} E(x) =-\infty$;
			\item $E(-N_1)<0$.
			\item $E(0)>0$;
			\item $E(N_2-N_1-2)>0$;
		\end{enumerate}
		Indeed the above implies that the cubic polynomial $E$ has two negative roots and one positive root that is strictly greater than $N_2-N_1-2$, thus for any $x\in[0,N_2-N_1-1]$, $E(x)>0$ and consequently $D(x)>0$.
		
		\begin{enumerate}[A.]
			\item It is clear from the definition of the coefficient $c_3$ above. 
			
			\item 
			By taking $N_2\geq N_1+2$ we have,
			\begin{align*}
				E(-N_1)
				&=c_0-c_1{N_1}+c_2{N_1}^2-c_3{N_1}^3 \\
				&=\left(-1+3{N_1}-{N_1}^2\right)+\left(-3+{N_1}+{N_1}^2\right){N_2}+\left(2-2{N_1}\right){N_2}^2 \\
				&\leq \left(-1+3{N_1}-{N_1}^2\right)+\left(1-{N_1}-{N_1}^2\right)N_2 \\
				&\leq 1+2{N_1}-4{N_1}^2-{N_1}^3 \\
				&<0.
			\end{align*}
			
			\item 
			$E(0)$ is given by
			\begin{align*}
				E(0)
				&= c_0\\
				&=-{N_1}^5N_2+{N_1}^5+{N_1}^4{N_2}^2-2{N_1}^4-2{N_1}^3{N_2}^2+5{N_1}^3N_2-{N_1}^3 \\
				&\qquad-{N_1}^2N_2+{N_1}^2-N_1{N_2}^2-2N_1N_2+4N_1+2{N_2}^2-3N_2-1 \\
				&=
				{N_2}^2({N_1}^4-2{N_1}^3-N_1+2)+N_2(-{N_1}^5+5{N_1}^3-{N_1}^2-2N_1-3)\\
				&\qquad+({N_1}^5-2{N_1}^4-{N_1}^3+{N_1}^2+4N_1-1).
			\end{align*}
			Suppose $N_1\geq 3$ and $N_2 \geq N_1+2$. the constant term and the coefficients of ${N_2}$ and ${N_2}^2$ are positive, negative and positive respectively. Thus, we have
			\begin{align*}
				c_0 &\geq {N_2}^2\left({N_1}^4-2{N_1}^3-N_1+2\right)+N_2\left(-{N_1}^5+5{N_1}^3-{N_1}^2-2N_1-3\right) \\
				&= N_2\left(N_2\left({N_1}^4-2{N_1}^3-N_1+2\right)+\left(-{N_1}^5+5{N_1}^3-{N_1}^2-2N_1-3\right)\right) \\
				&\geq  (N_1+2)\left({N_1}^4-2{N_1}^3-N_1+2\right)+\left(-{N_1}^5+5{N_1}^3-{N_1}^2-2N_1-3\right) \\
				&= (N_1+1)\left({N_1}^2-3N_1+1\right) \\
				&>0.
			\end{align*}
			
			\item 
			Suppose $N_1\geq3$ and $N_2\geq N_1+2\geq5$. Calculating $E(N_2-N_1-2)$ we get
			{
				\begin{align*}
					E(N_2-N_1-2)
					&=c_0+c_1(N_2-N_1-2)+c_2(N_2-N_1-2)^2+c_3(N_2-N_1-2)^3 \\
					% &=9+13{N_1}-11{N_1}^2+(-16-12 {N_1} +14 {N_1}^2) {N_2}+(8+4 {N_1}-6 {N_1}^2){N_2}^2 +(-1-{N_1}+{N_1}^2){N_2}^3 \\
					&=(9-16{N_2}+8{N_2}^2-{N_2}^3)\\
					&\qquad+(13-12{N_2}+4{N_2}^2-{N_2}^3){N_1}+(-11+14{N_2}-6{N_2}^2+{N_2}^3){N_1}^2
				\end{align*}
			}
			Now we consider the polynomial below for $x\geq3$ and $y\geq5$ 
			\begin{align*}
				P(x, y) = a_y x^2+b_y x+c_y,
			\end{align*}
			where
			\begin{align*}
				a(y) = -11+14y-6y^2+y^3, &&
				b(y) = 13-12y+4y^2-y^3, &&
				c(y) = 9-16y+8y^2-y^3.
			\end{align*}
			Note that $a$ is increasing and that for $y\geq5$,
			\begin{align*}
				a(y)=(y-2)^3+2(y-2)+1\geq34.
			\end{align*}
			Thus, for any fixed $y\geq5$, $P(x, y)$ is quadratic in $x $ and forms an upward parabola. The partial derivative of $P$ with respect to $x$ is given by
			\begin{align*}
				\frac{\partial P}{\partial x}(x,y)&=2a(y)x+b(y).
			\end{align*}
			Evaluating at $x=3$, we see that
			\begin{align*}
				\frac{\partial P}{\partial x}(3,y)
				&=5y^3-32y^2+72y-53 \\
				&=5\left(y-\frac{32}{15}\right)^3+\frac{56}{15}\left(y-\frac{32}{15}\right)+\frac{2369}{675}>0.
			\end{align*}
			This tells us that for fixed $y\geq 5$, $P(x,y)$ is increasing on $x\in[3,\infty)$. Hence, for any $x\geq3$ and $y\geq5$
			\begin{align*}
				P(x,y)
				&\geq P(3,y) \\
				&=(y-3) (5 y^2-19 y+17) \\
				&=(y-3)\left(5 (y-5)^2+31 (y-5)+47\right) \\
				&>0.
			\end{align*}
			
		\end{enumerate}
	\end{proof}

	\begin{proof}[Proof of Lemma~\ref{lemma: tau D}]
		Let $N_1,N_2\in\mathbb{N}\setminus\{1\}$ with $N_1<N_2$. To show that $f_1$ is decreasing, we write for any $x\geq 0$,
		\begin{align*}
			f_1(x)
			&=\frac{N_1+x\rangle-\langle N_1+x}{\langle N_1+x-N_1+x+1\rangle}
			=\frac{N_1\rangle-\langle N_2}{\displaystyle
				(1-N_1\rangle)\frac{N_1+x-1}{N_1+x+1}+\langle N_2}
			,
		\end{align*}
		it remains to notice that the following expression is strictly increasing on $[0,+\infty)$ with respect to $x$:
		\begin{align*}
			\frac{N_1+x-1}{N_1+x+1}=1-\frac{2}{{N_1+x+1}}.
		\end{align*}
	\end{proof}
	
	\begin{proof}[Proof of Lemma~\ref{lemma: tau E}]
		It suffices to prove that $D(N_2-N_1-1)\leq0$. By $2\leq N_1<N_2$, we have
		\begin{align*}
			D(N_2-N_1-1)
			&=\frac{1}{N_2}\left(1+\frac{N_1\rangle}{N_2-1}\right)-\langle N_2-\frac{N_1\rangle-\langle N_2}{(N_2-1)(N_2-2)} \\
			&=\frac{4N_1-2{N_1}^{2}-4N_2+2{N_1}^{2}N_2+2{N_2}^{2}-2N_1{N_2}^{2}}{({N_1}^2-{N_1}-1){N_2}({N_2}-1)({N_2}-2)({N_2}^2-{N_2}-1)} \\
			&<\frac{(N_1-N_2)(N_1N_2-N_1-N_2)}{({N_1}^2-{N_1}-1){N_2}({N_2}-1)({N_2}-2)({N_2}^2-{N_2}-1)} \\
			&<0
			.
		\end{align*}
		Hence, $f_0(N_2-N_1-1) \leq f_1(N_2-N_1-1)$, which completes the proof of Lemma \ref{lemma: tau E}, which in turn completes the proof of Proposition \ref{prop: Stepwise Complete tau}.
	\end{proof}
	
	\section{Proofs on Congruence}\label{sectproofs}
	In this section, we prove the main results on whether sums of L\"uroth sets are congruent to $\mathbb{R}$ mod 1, namely Theorems~\ref{thm: 1},~\ref{thm: 2}, Corollary~\ref{coro: 3}, and Theorem~\ref{thm: 4}.
	
	\subsection{Proof of Theorem~\ref{thm: 1}}\label{subsectproof1}
	In Section \ref{sect: Constructions}, we construct the set $L_{\leq 3}$ under the SCC and we have that,
	\begin{align*}
		g_3
		&= \frac{\langle2-3\rangle}{2\rangle-\langle3}=\frac{7/10-1/2}{1-2/5}
		&&= \frac13; \\
		h_3
		&=\min{\left\{\frac{3\rangle-\langle3}{2\rangle-\langle3},\frac{2\rangle-\langle2}{2\rangle-\langle3}\right\}}
		=\min{\left\{\frac{1/2-2/5}{1-2/5},\frac{1-7/10}{1-2/5}\right\}}
		&&= \frac{1}{6}.
	\end{align*}
	Similarly calculating $g$ and $h$ for $L_{\leq 4}$ and $L_{\leq 5}$ yields that,
	\begin{align*}
		g_4
		= \frac{\langle3-4\rangle}{3\rangle-\langle4}
		= \frac{1}{5}; &&
		h_4
		= \frac{\langle4-4\rangle}{3\rangle-\langle4}
		= \frac{4}{15}; \\
		g_5
		= \frac{\langle4-5\rangle}{4\rangle-\langle5}
		= \frac{1}{7} ; &&
		h_5
		= \frac{\langle5-5\rangle}{4\rangle-\langle5}
		= \frac{9}{28}.
	\end{align*}
	To complete the proof, we apply Propositions~\ref{prop: hlavka3} and \ref{prop: hlavka10}.
	
	To prove that $L_{\leq 3}+L_{\leq 5}\equiv\mathbb{R}\mod1$, we notice that $\lvert I_3 \rvert = 3/5$ and $\lvert I_5\rvert = 15/19$. We verify the three conditions in Proposition~\ref{prop: hlavka3}:
	\begin{align*}
		g_3 g_5
		= \frac{1}{21}
		\leq \frac{3}{56}
		= h_3h_5, &&
		g_3 \lvert I_3 \rvert
		= \frac{1}{5}
		\leq \frac{15}{19}
		= \lvert I_5 \rvert,  &&
		g_5 \lvert I_5 \rvert 
		= \frac{15}{133}
		\leq \frac{3}{5}
		= \lvert I_3 \rvert.
	\end{align*}
	By Proposition~\ref{prop: hlavka3}, we see that $L_{\leq 3}+L_{\leq 5}$ is an interval of length $132/95\geq1$, as
	\begin{align*}
		L_{\leq 3}+L_{\leq 5}
		= I_3+I_5
		= \left[\frac{2}{5},1\right]+\left[\frac{4}{19},1\right] = \left[\frac{58}{95},2\right].
	\end{align*}
	
	To prove that $L_{\leq 4}+L_{\leq 4}\equiv\mathbb{R}\mod1$, we notice that $\lvert I_4 \rvert = 8/11$. We verify the conditions in Proposition~\ref{prop: hlavka3}:
	\begin{align*}
		g_4 g_4
		= \frac{1}{25}
		\leq \frac{16}{225}
		= h_4h_4, &&
		g_4 \lvert I_4 \rvert 
		= \frac{8}{55}
		\leq \frac{8}{11}
		= \lvert I_4 \rvert.
	\end{align*}
	By Proposition~\ref{prop: hlavka3}, we see that $L_{\leq 4}+L_{\leq 4}$ is an interval of length $16/11\geq1$, as
	\begin{align*}
		L_{\leq 4}+L_{\leq 4}
		= I_4+I_4
		= \left[\frac{3}{11},1\right]+\left[\frac{3}{11},1\right] = \left[\frac{6}{11},2\right].
	\end{align*}
	
	To prove that $L_{\leq 3}+L_{\leq 3}+L_{\leq 3}\equiv\mathbb{R}\mod1$, we apply Proposition~\ref{prop: hlavka10} with $g_3$ and $h_3$ obtained above. The first condition is clearly satisfied since all the $I_{\mathcal{A}}$'s are the same and $h$ is always not greater than 1. For the second condition, we have
	\begin{align*}
		g_3+h_3
		=3h_3
		=\frac{1}{2}.
	\end{align*}
	Proposition~\ref{prop: hlavka10} then tells us that $L_{\leq 3}+L_{\leq 3}+L_{\leq 3}$ is an interval of length $9/5\geq 1$, namely
	\begin{align*}
		L_{\leq 3}+L_{\leq 3}+L_{\leq 3}
		=I_3+I_3+I_3
		= \left[\frac{2}{5},1\right]+\left[\frac{2}{5},1\right]+\left[\frac{2}{5},1\right]
		= \left[\frac{6}{5},3\right].
	\end{align*}
	
	Next, we prove $L_{\leq 3}+L_{\leq 3}\not\equiv\mathbb{R}\mod1$. Notice that at Level 1 of SCC for $L_{\leq 3}$, we have
	\begin{align*}
		L_{\leq 3}
		\subset[\langle3,3\rangle]\cup[\langle2,2\rangle]
		=\left[\frac{2}{5},\frac{1}{2}\right]\cup\left[\frac{7}{10},1\right].
	\end{align*}
	Thus, we have
	\begin{align*}
		L_{\leq 3}+L_{\leq 3}\subset \left[\frac{4}{5},1\right]\cup\left[\frac{11}{10},2\right];
	\end{align*}
	hence,
	\begin{align*}
		L_{\leq 3}+L_{\leq 3}+\mathbb{Z}\subset\left[\frac{1}{10},1\right]+\mathbb{Z}.
	\end{align*}
	In particular, $L_{\leq 3}+L_{\leq 3}\not\equiv\mathbb{R}\mod1$.
	
	If we go down further levels in SCC we reveal more gaps in $L_{\leq 3}+L_{\leq 3}$. Note that
	\begin{align*}
		[\dot{3}]     =\frac{ 2}{ 5}, &&
		[3,3,\dot{2}] =\frac{ 5}{12}, &&
		[3,2,\dot{3}] =\frac{ 9}{20}, &&
		[3,\dot{2}]   =\frac{ 1}{ 2}; \\
		[2,\dot{3}]   =\frac{ 7}{10}, &&
		[2,3,\dot{2}] =\frac{ 3}{ 4}, &&
		[2,2,\dot{3}] =\frac{17}{20}, &&
		[\dot{2}]     =1. 
	\end{align*}
	Thus at Level 2 of the SCC for $L_{\leq 3}$, we have
	\begin{align*}
		L_{\leq 3}
		\subseteq
		\left[\frac{2}{5},\frac{5}{12}\right]\cup
		\left[\frac{9}{20},\frac{1}{2}\right]\cup
		\left[\frac{7}{10},\frac{3}{4}\right]\cup
		\left[\frac{17}{20},1\right],
	\end{align*}
	and hence,
	\begin{align*}
		L_{\leq 3}+L_{\leq 3}
		\subseteq
		\left[\frac{4}{5},\frac{5}{6}\right]\cup
		\left[\frac{17}{20},1\right]\cup
		\left[\frac{11}{10},\frac{3}{2}\right]\cup
		\left[\frac{31}{20},2\right].
	\end{align*}
	This in turn shows that
	\begin{align*}
		L_{\leq 3}+L_{\leq 3}+\mathbb{Z}
		&\subseteq
		\left[\frac{4}{5},\frac{5}{6}\right]\cup
		\left[\frac{17}{20},1\right]\cup
		\left[\frac{1}{10},\frac{1}{2}\right]\cup
		\left[\frac{11}{20},1\right]+\mathbb{Z} \\
		&=\left[\frac{1}{10},\frac{1}{2}\right]\cup
		\left[\frac{11}{20},1\right]+\mathbb{Z}
	\end{align*}
	hence revealing another gap in $L_{\leq 3}+L_{\leq 3}$.
	
	This also provides an upper bound for the lengths of the potential intervals contained in $L_{\leq3}+L_{\leq 3}$, but does not indicate whether it contains an interval or not. Figure~\ref{fig: L_3 L_3} shows where these gaps occur in $L_{\leq 3}+L_{\leq 3}$, as well as indicating where the gaps lie in $[0,1]$.
	
	\newlength{\lengthLuroth}
	\begin{figure}
		\setlength{\lengthLuroth}{.48\textwidth}
		\centering
		\begin{subfigure}[b]{.49\textwidth}
			\centering
			\begin{tikzpicture}
				\path[fill=blue!10!white] 
				plot[samples=10,domain=0:0.1] (0,{\x*\lengthLuroth}) -- 
				plot[samples=10,domain=0:0.1] ({\x*\lengthLuroth},{\lengthLuroth*(0.1-\x)});
				\path[fill=blue!10!white] 
				plot[samples=10,domain=0:0.1] ({\lengthLuroth},{\lengthLuroth*(0.1-\x)}) -- 
				plot[samples=10,domain=0:0.1] ({\x*\lengthLuroth},{\lengthLuroth});
				\draw[loosely dashed, blue] (0,{1/10*\lengthLuroth}) -- ({1/10*\lengthLuroth},0);
				\draw[loosely dashed, blue] ({1/10*\lengthLuroth},{1*\lengthLuroth}) -- ({1*\lengthLuroth},{1/10*\lengthLuroth});
				\draw[loosely dashed, blue] ({0*\lengthLuroth},{1*\lengthLuroth}) -- ({1*\lengthLuroth},{0*\lengthLuroth}) node[pos=0.75,below, blue, font=\small,sloped] {$(0,1/10)+\mathbb{Z}$};
				\foreach \xa/\xb in {{2/5}/{1/2}, {7/10}/{1}} {
					\foreach \ya/\yb in {{2/5}/{1/2}, {7/10}/{1}} {
						\fill[black] ({\xa*\lengthLuroth},{\ya*\lengthLuroth}) rectangle ({\xb*\lengthLuroth},{\yb*\lengthLuroth});
					}
				}
				\draw[black] (0,0) rectangle ++ (\lengthLuroth,\lengthLuroth);
			\end{tikzpicture}
			\caption{Level 1}
		\end{subfigure}
		\hfill
		\begin{subfigure}[b]{.49\textwidth}
			\centering
			\begin{tikzpicture}
				\path[fill=blue!10!white] 
				plot[samples=10,domain=0:0.1] (0,{\x*\lengthLuroth}) -- 
				plot[samples=10,domain=0:0.1] ({\x*\lengthLuroth},{\lengthLuroth*(0.1-\x)});
				\path[fill=blue!10!white] 
				plot[samples=10,domain=0:0.1] ({\lengthLuroth},{\lengthLuroth*(0.1-\x)}) -- 
				plot[samples=10,domain=0:0.1] ({\x*\lengthLuroth},{\lengthLuroth});
				\draw[loosely dashed, blue] (0,{1/10*\lengthLuroth}) -- ({1/10*\lengthLuroth},0);
				\draw[loosely dashed, blue] ({1/10*\lengthLuroth},{1*\lengthLuroth}) -- ({1*\lengthLuroth},{1/10*\lengthLuroth});
				\draw[loosely dashed, blue] ({0*\lengthLuroth},{1*\lengthLuroth}) -- ({1*\lengthLuroth},{0*\lengthLuroth});
				\path[fill=red!10!white] 
				plot[samples=10,domain=.5:.55] ({\lengthLuroth},{\lengthLuroth*\x}) -- 
				plot[samples=10,domain=.5:.55] ({(1.05-\x)*\lengthLuroth},{\lengthLuroth});
				\path[fill=red!10!white] 
				plot[samples=10,domain=.5:.55] (0,{\lengthLuroth*(\x)}) -- 
				plot[samples=10,domain=.5:.55] ({(1.05-\x)*\lengthLuroth},0);
				\draw[loosely dotted, red] (0,{1/2*\lengthLuroth}) -- ({1/2*\lengthLuroth},0) node[pos=0.5,below, red, font=\small,sloped] {$(1/2,11/20)+\mathbb{Z}$};
				\draw[loosely dotted, red] (0,{11/20*\lengthLuroth}) -- ({11/20*\lengthLuroth},0);
				\draw[loosely dotted, red] ({1/2*\lengthLuroth},{1*\lengthLuroth}) -- ({1*\lengthLuroth},{1/2*\lengthLuroth});
				\draw[loosely dotted, red] ({11/20*\lengthLuroth},{1*\lengthLuroth}) -- ({1*\lengthLuroth},{11/20*\lengthLuroth});
				\foreach \xa/\xb in {{2/5}/{5/12}, {9/20}/{1/2}, {7/10}/{3/4}, {17/20}/{1}} {
					\foreach \ya/\yb in {{2/5}/{5/12}, {9/20}/{1/2}, {7/10}/{3/4}, {17/20}/{1}} {
						\fill[black] ({\xa*\lengthLuroth},{\ya*\lengthLuroth}) rectangle ({\xb*\lengthLuroth},{\yb*\lengthLuroth});
					}
				}
				\draw[black] (0,0) rectangle ++ (\lengthLuroth,\lengthLuroth);
			\end{tikzpicture}
			\caption{Level 2}
		\end{subfigure}
		
		\caption{Different Levels of $L_{\leq 3}$ under SCC product with itself on $[0,1]^2$}
		\label{fig: L_3 L_3}
	\end{figure}
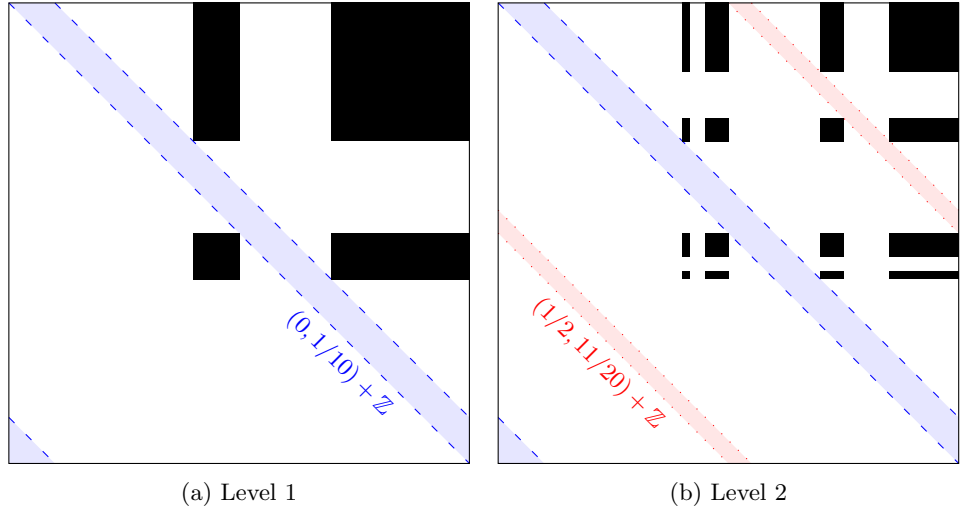

	Finally, we prove that $L_{\leq 3}+L_{\leq 4}$ is not an interval. Note that Level 3 of the SCC for $L_{\leq 3}$ is given by:
	\begin{align*}
		\left[\frac{2}{5},\frac{29}{72}\right]\cup
		\left[\frac{49}{120},\frac{5}{12}\right]\cup
		\left[\frac{9}{20},\frac{11}{24}\right]\cup
		\left[\frac{19}{40},\frac{1}{2}\right]\cup
		\left[\frac{7}{10},\frac{17}{24}\right]\cup
		\left[\frac{29}{40},\frac{3}{4}\right]\cup
		\left[\frac{17}{20},\frac{7}{8}\right]\cup
		\left[\frac{37}{40},1\right].
	\end{align*}
	Level 4 of the SCC for $L_{\leq 4}$ is given by:
	\begin{align*}
		\left[\frac{3}{11}, \frac{5}{18}\right]\cup
		\left[\frac{223}{792}, \frac{7}{24}\right]\cup
		\left[\frac{10}{33}, \frac{5}{16}\right]\cup
		\left[\frac{7}{22}, \frac{1}{3}\right]\cup
		\left[\frac{25}{66}, \frac{7}{18}\right]\cup
		\left[\frac{157}{396}, \frac{5}{12}\right]\cup
		\left[\frac{29}{66}, \frac{11}{24}\right]\cup
		\left[\frac{31}{66}, \frac{1}{2}\right]\cup \\
		\left[\frac{7}{11}, \frac{31}{48}\right]\cup
		\left[\frac{43}{66}, \frac{2}{3}\right]\cup
		\left[\frac{91}{132}, \frac{17}{24}\right]\cup
		\left[\frac{95}{132}, \frac{3}{4}\right]\cup
		\left[\frac{9}{11}, \frac{5}{6}\right]\cup
		\left[\frac{223}{264}, \frac{7}{8}\right]\cup
		\left[\frac{10}{11}, \frac{15}{16}\right]\cup
		\left[\frac{21}{22}, 1\right].
	\end{align*}
	Consequently, we get
	\begin{align*}
		L_{\leq 3}+L_{\leq 4}
		\subseteq
		\left[\frac{37}{55},\frac{49}{72}\right] \cup
		\left[\frac{899}{1320},2\right]
		,
	\end{align*}
	and for each interval, there is an element belonging to $L_{\leq 3}+L_{\leq 4}$, namely
	\begin{align*}
		[\dot3]+[\dot4]=\frac{2}{5}+\frac{3}{11}=\frac{37}{55}\in L_{\leq 3}+L_{\leq 4},
	\end{align*}
	and
	\begin{align*}
		[\dot2]+[\dot2]=1+1=2\in L_{\leq 3}+L_{\leq 4}.
	\end{align*}
	Figure~\ref{fig: L_3 L_4} in appendix illustrates this gap in $L_{\leq 3}+L_{\leq 4}$.
	
	\subsection{Proof of Theorem~\ref{thm: 2}}
	The following result is useful to prove Theorem~\ref{thm: 2}. It basically says that $L_{\geq k}$ can be estimated arbitrarily well by $L_{k,N}$ for sufficiently large $N$. 
	\begin{proposition}\label{prop: Luroth Cusick--Lee Type Helper Improved}
		For any $k\in\mathbb{N}\setminus\{1\}$ and $\varepsilon>0$, there exists $N_{k,\varepsilon}\in\mathbb{N}$ such that
		\begin{align*}
			\gamma_{k,N_{k,\varepsilon}}
			&>\frac{k-1}{k(k-1)-1}-\varepsilon, \\
			\left\lvert{I_{k,N_{k,\varepsilon}}}\right\rvert
			&>\frac{k-1}{k(k-1)-1}-\varepsilon, \\
			\left\lvert{O_{k,N_{k,\varepsilon}}}\right\rvert
			&<\frac{k-2}{k(k^2-2)-1}+\varepsilon
			.
		\end{align*}
	\end{proposition}
	\begin{proof}
		Pick any $k\in\mathbb{N}\setminus\{1\}$ and $\varepsilon>0$. Define
		\begin{align*}
			N_{k,\varepsilon}
			\coloneqq \max{\left\{k,\frac{1}{\varepsilon}\left(\frac{k-1}{k(k-1)-1}+1\right)+1\right\}}+1>k.
		\end{align*}
		Then for $N_1=k$ and $N_2=N=N_{k,\varepsilon}$, using SCC we have,
		\begin{align*}
			\left\lvert{I_{k,N}}\right\rvert
			&=k\rangle-\langle N \\
			&=\frac{k-1}{k(k-1)-1}-\frac{N-1}{N(N-1)-1} \\
			&\geq \frac{k-1}{k(k-1)-1}-\frac{1}{N-1} \\
			&>\frac{k-1}{k(k-1)-1}-\varepsilon;
		\end{align*}
		and
		\begin{align*}
			\left\lvert{O_{k,N}}\right\rvert
			&=\langle k-k+1\rangle \\
			&=\left(\frac{1}{k}+\frac{1}{k(k-1)}\frac{N-1}{N(N-1)-1}\right)-\left(\frac{1}{k+1}+\frac{1}{(k+1)k}\frac{k-1}{k(k-1)-1}\right) \\
			&=\left(\frac{1}{k}-\frac{1}{k+1}-\frac{1}{(k+1)k}\frac{k-1}{k(k-1)-1}\right)+\frac{1}{k(k-1)}\frac{N-1}{N(N-1)-1} \\
			&\leq\left(\frac{1}{k}-\frac{1}{k+1}-\frac{1}{(k+1)k}\frac{k-1}{k(k-1)-1}\right)+\frac{1}{N-1} \\
			&<\frac{k-2}{k(k^2-2)-1}+\varepsilon
			.
		\end{align*}
		By applying Proposition \ref{prop: Stepwise Complete tau}, we get the thickness:
		\begin{align*}
			\tau_{k,N}
			=\frac{N\rangle-\langle N}{\langle(N-1)-N\rangle}
			.
		\end{align*}
		The numerator is given by,
		\begin{align*}
			N\rangle-\langle N
			&=\left(\frac{1}{N}+\frac{1}{N(N-1)}\frac{k-1}{k(k-1)-1}\right)-\frac{N-1}{N(N-1)-1} \\
			&=\frac{1}{N(N-1)}\left(\frac{k-1}{k(k-1)-1}-\frac{N-1}{N(N-1)-1}\right)
			,
		\end{align*}
		and the denominator is given by,
		\begin{align*}
			\langle(N-1)&-N\rangle \\
			&=\left(\frac{1}{N-1}+\frac{1}{(N-2)}\frac{1}{N(N-1)-1}\right)-\left(\frac{1}{N}+\frac{1}{N(N-1)}\frac{k-1}{k(k-1)-1}\right) \\
			&=\frac{1}{N(N-1)}\left(\frac{N^3-2N^2+2}{(N-2)(N^2-N-1)}-\frac{k-1}{k(k-1)-1}\right).
		\end{align*}
		Hence we have
		\begin{align*}
			\tau_{k,N}
			=\left.\left(\frac{k-1}{k(k-1)-1}-\frac{N-1}{N(N-1)-1}\right)\middle/\left(\frac{N^3-2N^2+2}{(N-2)(N^2-N-1)}-\frac{k-1}{k(k-1)-1}\right)\right.
			,
		\end{align*}
		Lastly we can calculate $\gamma$,
		\begin{align*}
			\gamma_{k,N}
			&=\frac{1}{1+1/\tau_{k,N}} \\
			&=\left.\left(\frac{k-1}{k(k-1)-1}-\frac{N-1}{N(N-1)-1}\right)\middle/\left(\frac{N^3-2N^2+2}{(N-2)(N^2-N-1)}-\frac{N-1}{N(N-1)-1}\right)\right. \\
			&=\left(\frac{k-1}{k(k-1)-1}-\frac{N-1}{N(N-1)-1}\right)\frac{N^3-3N^2+N+2}{N^3-3N^2+3N} \\
			&>\left(\frac{k-1}{k(k-1)-1}-\frac{1}{N-1}\right)\left(1-\frac{1}{N-1}\right) \\
			&>\frac{k-1}{k(k-1)-1}-\frac{1}{N-1}\left(\frac{k-1}{k(k-1)-1}+1\right) \\
			&>\frac{k-1}{k(k-1)-1}-\varepsilon
			,
		\end{align*}
		which completes the proof.
	\end{proof}
	
	\begin{proof}[Proof of Theorem~\ref{thm: 2}]
		It suffices to prove that there exists $\varepsilon>0$ such that
		\begin{align*}
			L_{k_1,N_{k_1,\varepsilon}}+
			L_{k_2,N_{k_2,\varepsilon}}+\cdots+
			L_{k_n,N_{k_n,\varepsilon}}\equiv\mathbb{R}\mod1.
		\end{align*}
		We first require
		\begin{align*}
			0<\varepsilon<\frac{1}{n}\left(\sum_{j=1}^{n}\frac{k_j-1}{k_j(k_j-1)-1}-1\right).
		\end{align*}
		then we get, by Proposition~\ref{prop: Luroth Cusick--Lee Type Helper Improved}, that
		\begin{align*}
			S_{\gamma}
			=\sum_{j=1}^n\gamma_{k_j,N_{k_j,\varepsilon}}>1, &&
			\sum_{j=1}^n\left\lvert I_{k_j,N_{k_j,\varepsilon}}\right\rvert=\sum_{j=1}^n\operatorname{diam}{L_{k_j,N_{k_j,\varepsilon}}}>1
			.
		\end{align*}
		It remains to verify the conditions \eqref{eq: (12)} and \eqref{eq: (13)}, where $C_{\mathcal{A}_i}=L_{k_{n+1-i},N_{k_{n+1-i},\varepsilon}}$. For \eqref{eq: (12)}, for any $i<j$,
		\begin{align*}
			\left\lvert{I_{k_i,N_{k_i,\varepsilon}}}\right\rvert
			>\frac{k_i-1}{k_i(k_i-1)-1}-\varepsilon
			>\frac{k_j-2}{k_j({k_j}^2-2)-1}+\varepsilon
			>
			\left\lvert{O_{k_j,N_{k_j,\varepsilon}}}\right\rvert
			,
		\end{align*}
		where we also assume
		\begin{align*}
			0
			<\varepsilon
			<\frac{1}{2}\min_{i<j}{\left\{\frac{k_i-1}{k_i(k_i-1)-1}-\frac{k_j-2}{k_j({k_j}^2-2)-1}\right\}}
			.
		\end{align*}
		For \eqref{eq: (13)}, we need \eqref{eq: (4)}, for any $j=2,\dots,n$,
		\begin{align*}
			\left\lvert{I_{k_j,N_{k_j,\varepsilon}}}\right\rvert
			&>\frac{k_j-1}{k_j(k_j-1)-1}-\varepsilon \\
			&>\frac{1}{k_j}
			\geq\frac{1}{{{k_{j-1}}^2+2k_{j-1}+2}} \\
			&>\frac{k_{j-1}-2}{k_{j-1}({k_{j-1}}^2-2)-1}+\varepsilon
			>
			\left\lvert{O_{k_{j-1},N_{k_{j-1},\varepsilon}}}\right\rvert
			,
		\end{align*}
		here we are furthermore requiring that
		\begin{align*}
			0<\varepsilon
			<\min_{j\geq2}\left\{\frac{k_j-1}{k_j(k_j-1)-1}-\frac{1}{k_j},\frac{1}{{{k_{j-1}}^2+2k_{j-1}+2}}-\frac{k_{j-1}-2}{k_{j-1}({k_{j-1}}^2-2)-1}\right\}
			.
		\end{align*}
		The result then follows by applying Proposition~\ref{prop: astels}.
	\end{proof}
	
	\subsection{Proof of Corollary~\ref{coro: 3}}\label{subsectproof2}
	
	One can observe that Corollary~\ref{coro: 3} follows directly from Theorem~\ref{thm: 2}, by taking $k_j=k$ for each $j\in\{1,\ldots,n\}$. Alternatively, it is possible to prove it directly, by applying Proposition~\ref{prop: astels}. For $k=3$, it is even possible to prove it by applying Proposition~\ref{prop: hlavka10}.
	
	By Proposition~\ref{prop: Stepwise Complete tau} and the definition of $\gamma$, we see that for $L_{3,16}$,
	\begin{align*}
		\gamma_{3,16}=\frac{2821}{8440}>\frac{1}{3}.
	\end{align*}
	By Proposition~\ref{prop: astels} since $3\gamma_{3,16}>1$, we have
	\begin{align*}
		3 L_{3,16}
		=3 I_{3,16}
		=3 \left[\frac{15}{239},\frac{2}{5}\right]
		=\left[\frac{45}{239},\frac{6}{5}\right],
	\end{align*}
	which is of length ${1209}/{1195}>1$. Therefore,
	\begin{align*}
		3 L_{3,16}\equiv\mathbb{R}\mod1,
	\end{align*}
	and consequently $3 L_{\geq3}\equiv\mathbb{R}\mod1$. Similarly,
	\begin{align*}
		4 L_{4,45}\equiv\mathbb{R}\mod1, &&
		5 L_{5,96}\equiv\mathbb{R}\mod1, &&
		6 L_{6,175}\equiv\mathbb{R}\mod1, 
	\end{align*}
	and so on. In general, for any $k\in\mathbb{N}\setminus\{1\}$, it can be verified that,
	\begin{align*}
		k L_{k,k^3}\equiv\mathbb{R}\mod1.
	\end{align*}
	By the formula for $\tau_{k,k^3}$ derived in Proposition \ref{prop: Stepwise Complete tau}, we have the value of $k\gamma_{k,k^3}$ being at least 1:
	\begin{align*}
		k\gamma_{k,k^3}
		&=\frac{k(k-1)(k+1)(k^3-2)(k^4-k^3-k+2)}{k^2(k^2-k-1)(k^6-3k^3+3)} \\
		&=1+\frac{k^6+k^5-5k^3-k^2+k+4}{k^9-k^8-k^7-3 k^6+3k^5+3k^4+3k^3-3 k^2-3k} \\
		&\geq 1.
	\end{align*}
	
	Lastly for $k=3$, we can apply Proposition~\ref{prop: hlavka10}, by considering an unordered construction of $L_{3,26}$. For the unordered construction, given in Figure~\ref{fig: L_3_26} in appendix, we have
	\begin{align*}
		g_{3,26}
		=\frac{\langle25-26\rangle}{25\rangle-\langle26}
		=\frac{24989}{54314}, &&
		h_{3,26}
		=\frac{6\rangle-\langle6}{\langle5-6\rangle}
		= \frac{391}{1689},
	\end{align*}
	where the $g$- and $h$- values are obtained from Level 6 and 4 respectively. Note that
	\begin{align*}
		g_{3,26}+h_{3,26}
		= \frac{63443195}{91736346}
		\leq
		\frac{391}{563}
		= 3 h_{3,26}.
	\end{align*}
	By Proposition~\ref{prop: hlavka10}, we see that
	\begin{align*}
		3 L_{3,26}
		= 3 I_{3,26}
		= 3\left[\frac{25}{649}, \frac{2}{5}\right]
		= \left[\frac{75}{649}, \frac{6}{5}\right]
		,
	\end{align*}
	which is of length ${3519}/{3245}> 1$. Hence,
	\begin{align*}
		3 L_{3,26} \equiv \mathbb{R} \mod 1,
	\end{align*}
	and consequently $3 L_{\geq 3} \equiv \mathbb{R} \mod 1$.
	
	\subsection{Proof of Theorem~\ref{thm: 4}}
	Let $k\in\mathbb{N}\setminus\{1\}$. By Proposition~\ref{prop: Stepwise Complete tau}, the thickness and the $\gamma$-value of $L_{\leq k+2}$ are
	\begin{align*}
		\tau_{k+2}=\frac{k^2}{k+1}, && \gamma_{k+2}=\frac{1}{1+1/\tau_{k+2}}=\frac{k^2}{k^2+k+1}.
	\end{align*}
	Let
	\begin{align*}
		\varepsilon
		=\min{\left\{\frac{2k}{k^4-k^2-2 k-1},\frac{2k^2}{k^4+2k^3-3k^2-4k-1},\frac{1}{2}\right\}}
		>0.
	\end{align*}
	By Proposition~\ref{prop: Luroth Cusick--Lee Type Helper Improved}, there exists $N_{k,\varepsilon}\in\mathbb{N}\setminus\{1\}$ such that
	\begin{align*}
		\gamma_{k,N_{k,\varepsilon}}>\frac{k-1}{k(k-1)-1}-\varepsilon, \\
		\left\lvert{I_{k,N_{k,\varepsilon}}}\right\rvert>\frac{k-1}{k(k-1)-1}-\varepsilon, \\
		\left\lvert{O_{k,N_{k,\varepsilon}}}\right\rvert<\frac{k-2}{k(k^2-2)-1}+\varepsilon.
	\end{align*}
	Now we observe, that
	\begin{align*}
		\gamma_{k+2}+\gamma_{k,N_{k,\varepsilon}}
		>\frac{k^4-k^2-1}{k^4-k^2-2 k-1}-\varepsilon
		\geq1,
	\end{align*}
	and
	\begin{align*}
		\left\lvert{I_{k+2}}\right\rvert=\frac{k (k+2)}{k^2+3 k+1}\geq\frac{k-2}{k(k^2-2)-1}+\frac{1}{2}\geq\left\lvert{O_{k,N_{k,\varepsilon}}}\right\rvert.
	\end{align*}
	By applying Proposition~\ref{prop: astels} we get,
	\begin{align*}
		L_{\leq k+2}+L_{k,N_{k,\varepsilon}}=I_{k+2}+I_{k,N_{k,\varepsilon}}.
	\end{align*}
	
	On the other hand,
	\begin{align*}
		\left\lvert{I_{k+2}}\right\rvert+\left\lvert{I_{k,N_{k,\varepsilon}}}\right\rvert
		&>\frac{k(k+2)}{k^2+3k+1}+\frac{k-1}{k(k-1)-1}-\varepsilon \\
		&=\frac{k^4+2k^3-k^2-4k-1}{k^4+2k^3-3k^2-4k-1}-\varepsilon \\
		&\geq1.
	\end{align*}
	Hence
	\begin{align*}
		L_{\leq k+2}+L_{k,N_{k,\varepsilon}}\equiv\mathbb{R}\mod1,
	\end{align*}
	and consequently $L_{\leq k+2}+L_{\geq k}\equiv\mathbb{R}\mod1$.
	%Still seems to work
	
	\section{Proofs on Dimension}
	In this section, we prove the main results on the Hausdorff dimensions of L\"uroth sets and sums of L\"uroth sets, namely Theorems~\ref{thm: dim A},~\ref{thm: dim B}, and~\ref{thm: dim C}.
	
	\subsection{Proof of Theorem~\ref{thm: dim A}}
	
	In the degenerate case, that is, at least one of $\mathcal{A}$ and $\mathcal{B}$ is a singleton, the result holds trivially. Indeed, without loss of generality, there exists $a\in\mathbb{N}\setminus\{1\}$ such that $\mathcal{A} = \{a\}$. Then we have,
	\begin{align*}
		L_{\mathcal{A}}=\left\{\frac{a-1}{(a-1)a-1}\right\},
	\end{align*}
	and consequently $\dim{L_{\mathcal{A}}}=0$. Hence, $\dim{(L_{\mathcal{A}}+L_{\mathcal{B}})}=\dim{L_{\mathcal{B}}}=\dim{L_{\mathcal{A}}}+\dim{L_{\mathcal{B}}}
	\leq1$.
	
	Suppose that both $\mathcal{A}$ and $\mathcal{B}$ are finite and not singletons. We aim to apply the result of Peres--Shmerkin~\cite[Theorem 2]{PERES_SHMERKIN_2009} on the Hausdorff dimension of the sumset of two iterative function systems (IFS).
	\begin{proposition}[Peres--Shmerkin {\cite[Theorem 2]{PERES_SHMERKIN_2009}}, 2009]\label{dimensionres}\label{prop: Peres--Shmerkin}
		Let $K$ and $K'$ be the attractors of the IFS's $\{r_i x+t_i\}_{i=1}^{n}$ and $\{r'_i x+t'_i\}_{i=1}^{n'}$ respectively. If there exist $j, j'$ such that $\log{\lvert r_j \rvert}/\log{\lvert r'_{j'} \rvert}$ is irrational, then
		\begin{align*}
			\dim{(K+K')} = \min{\left\{1,\dim{K}+\dim{K'}\right\}}.
		\end{align*}
	\end{proposition}
	
	Since $\mathcal{A}$ and $\mathcal{B}$ are finite, $L_{\mathcal{A}}$ and $L_{\mathcal{B}}$ can be expressed as IFS's:
	\begin{align*}
		L_{\mathcal{A}}=\bigcup_{d\in\mathcal{A}}f_d(L_{\mathcal{A}}), &&
		L_{\mathcal{B}}=\bigcup_{d\in\mathcal{B}}f_d(L_{\mathcal{B}})
		,
	\end{align*}
	where for any $d\in\mathbb{N}\setminus\{1\}$ and $x\in[0,1]$,
	\begin{align*}
		f_d(x)\coloneqq\frac{1}{d(d-1)}x+\frac{1}{d}.
	\end{align*}
	Since both $\mathcal{A}$ and $\mathcal{B}$ are not singletons, there exist $a\in\mathcal{A}$ and $b\in\mathcal{B}$ such that $a\ne b$. By a previous result by the second author~\cite{lee2024explicitupperboundsdecay}, we have that
	\begin{align*}
		\frac{\log(a(a-1))}{\log(b(b-1))}
	\end{align*}
	is irrational. By applying Proposition~\ref{prop: Peres--Shmerkin}, the proof of Theorem~\ref{thm: dim A} is completed.
	
	\subsection{Proof of Theorem~\ref{thm: dim B}}
	The proof applies a similar technique to that of Good. It suffices to prove the result for $k\geq16$, since for any $k\leq15$ we have $L_{\geq16}\subseteq L_{\geq k}$, so
	\begin{align*}
		\dim{L_{\geq k}}
		\geq\dim{L_{\geq 16}}
		>\frac{1}{2}+\frac{1}{2\log{16}}
		=\frac{1}{2}+\frac{1}{2\log{\max{\{16,k\}}}}.
	\end{align*}
	Let $k\in\mathbb{N}$. Suppose $k\geq16$. By the Intermediate Value Theorem, there exists $x_k>0$ such that $k^{-x_k}=x_k$. Let $\varepsilon>0$ and pick
	\begin{align*}
		N
		\coloneqq
		\max{\left\{
			k+1,
			\left\lceil{\left(x_k\left(k^\varepsilon-1\right)\right)}^{-1/(x_k-\varepsilon)}\right\rceil
			\right\}}.
	\end{align*}
	Define $F=F_{k,\varepsilon}:\mathbb{R}\to\mathbb{R}$ by, for any $x\in\mathbb{R}$,
	\begin{align*}
		F(x)\coloneqq k^{-x}-N^{-x}-x.
	\end{align*}
	Note that $F$ is continuous on $\mathbb{R}$, $F(x_k)<0$, and 
	\begin{align*}
		F(x_k-\varepsilon)
		&=k^{-x_k+\varepsilon}-N^{-x_k+\varepsilon}-(x_k-\varepsilon) \\
		&=x_k(k^\varepsilon-1)+\varepsilon-N^{-x_k+\varepsilon}\geq\varepsilon>0.
	\end{align*}
	By the Intermediate Value Theorem, there exists $x_{k,\varepsilon}\in(x_k-\varepsilon,x_k)$ such that $F(x_{k,\varepsilon})=0$. Thus, by writing $s_k=(x_k+1)/2\in\mathbb{R}$, there exists $s_{k,\varepsilon}>s_k-\varepsilon/2$ such that
	\begin{align*}
		\frac{1}{2s_{k,\varepsilon}-1}\left(k^{-(2s_{k,\varepsilon}-1)}-N^{-(2s_{k,\varepsilon}-1)}\right)=1,
	\end{align*}
	and $k^{-(2s_k-1)}=2s_k-1$. Thus, for any $\varepsilon>0$,
	\begin{align*}
		\sum_{d=k}^{N}\left(\frac{1}{d(d-1)}\right)^{s_{k,\varepsilon}}
		>\int_{k}^{N}\frac{\mathrm{d}x}{x^{2s_{k,\varepsilon}}}=1;
	\end{align*}
	and
	\begin{align*}
		\dim{L_{\geq k}}\geq\dim{L_{k,N}}>s_{k,\varepsilon}>s_k-\varepsilon/2.
	\end{align*}
	Since this $\varepsilon>0$ is arbitrary, we have
	\begin{align*}
		\dim{L_{\geq k}}\geq s_{k}.
	\end{align*}
	Since $k\geq16>e^e$, the equality $k^{-(2s_k-1)}=2s_k-1$ implies that
	\begin{align*}
		2s_k-1
		&>\frac{1}{\log{k}} \\
		s_k
		&>\frac{1}{2}+\frac{1}{2\log{k}}.
	\end{align*}
	
	\subsection{Proof of Theorem~\ref{thm: dim C}}
	It suffices to prove the following result.
	\begin{proposition}\label{prop: dim L_k,3k}
		For any $k\in\mathbb{N}\setminus\{1\}$, $\dim{L_{k,3k}}>1/2$.
	\end{proposition}
	\begin{proof}
		By, for example~\cite[Theorem 2]{lee2024explicitupperboundsdecay}, the Hausdorff dimension of $L_{k,3k}$ is the unique solution $s=s_{k,3k}\in[0,1]$ to 
		\begin{align*}
			\sum_{d=k}^{3k}\left(\frac{1}{d(d-1)}\right)^{s_{k,3k}}=1.
		\end{align*}
		Notice that
		\begin{align*}
			\sum_{d=k}^{3k}\left(\frac{1}{d(d-1)}\right)^{1}
			=\frac{2k+1}{3k(k-1)}
			<1.
		\end{align*}
		and
		\begin{align*}
			\sum_{d=k}^{3k}\left(\frac{1}{d(d-1)}\right)^{1/2}
			>\int_{k}^{3k}\frac{\mathrm{d}x}{x}
			=\log{3}
			>1.
		\end{align*}
		By the Intermediate Value Theorem, $s_{k,3k}>1/2$.
	\end{proof}
	\begin{proof}[Proof of Theorem~\ref{thm: dim C}]
		By Theorem~\ref{thm: dim A} and Proposition~\ref{prop: dim L_k,3k}, for any $k\in\mathbb{N}\setminus\{1\}$,
		\begin{align*}
			\dim{(L_{\geq k}+L_{\geq k})}
			&\geq\dim{\left(L_{k,3k}+L_{k,3k}\right)} \\
			&=\min{\left\{1,\dim{L_{k,3k}}+\dim{L_{k,3k}}\right\}} \\
			&=1.
		\end{align*}    
	\end{proof}
	
	\section*{Acknowledgements}
	The first author would like to express sincere gratitude to Professor Sanju Velani and the Department of Mathematics at the University of York for hosting the research environment change. Both authors also express their appreciation to Dr. Raphael Bennett-Tennenhaus for providing valuable and insightful feedback on the manuscript, and to the anonymous referees for their comments, which improved the clarity of the paper.

	%Appendix
	
	\appendix
	\section{}
	
	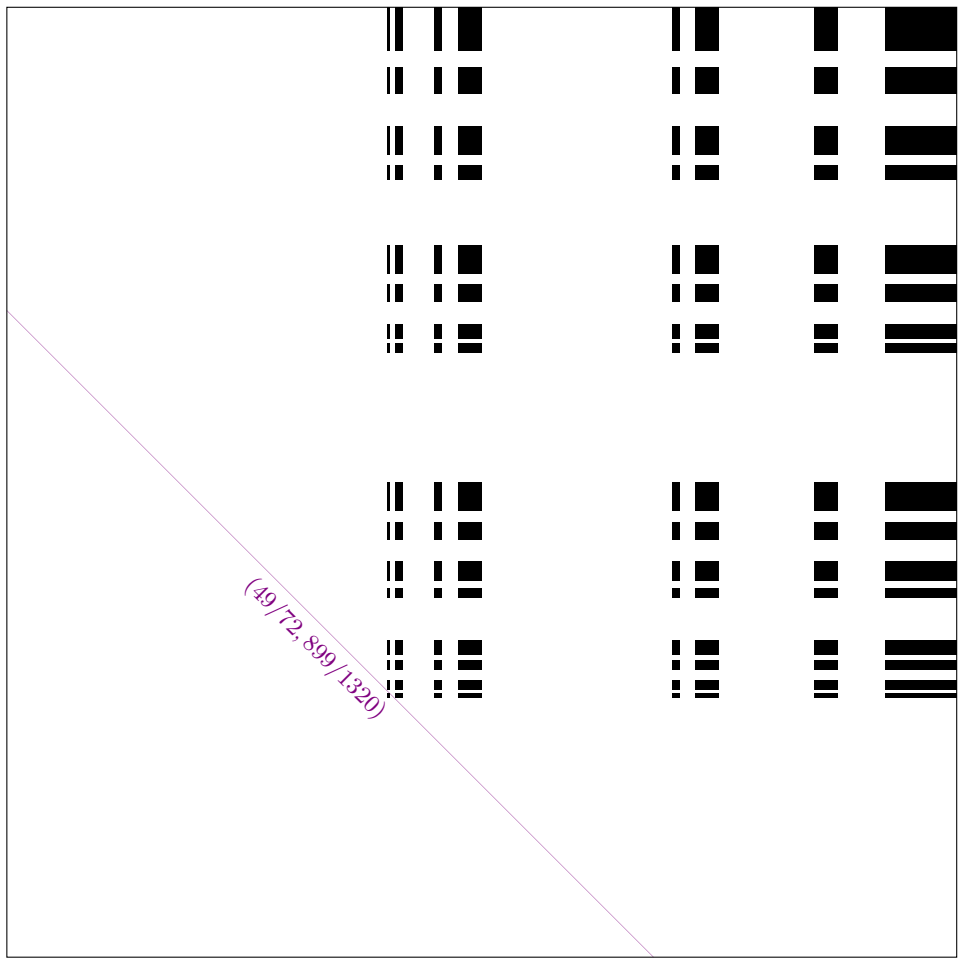
\begin{figure}[H]
		\centering
		\setlength{\lengthLuroth}{.99\textwidth}
		\begin{tikzpicture}
			\path[fill=violet] 
			plot[samples=10,domain=49/72:899/1320] (0,{\lengthLuroth*(\x)}) -- 
			plot[samples=10,domain=49/72:899/1320] ({(674/495-\x)*\lengthLuroth},0);
			\path (0,{337/495*\lengthLuroth}) -- ({337/495*\lengthLuroth},0) node[pos=0.5,below, violet, font=\small,sloped] {$(49/72,899/1320)$};
			\foreach \xa/\xb in {
				{2/5}/{29/72},
				{49/120}/{5/12},
				{9/20}/{11/24},
				{19/40}/{1/2},
				{7/10}/{17/24},
				{29/40}/{3/4},
				{17/20}/{7/8},
				{37/40}/1} {
				\foreach \ya/\yb in {
					{3/11}/{5/18},
					{223/792}/{7/24},
					{10/33}/{5/16},
					{7/22}/{1/3},
					{25/66}/{7/18},
					{157/396}/{5/12},
					{29/66}/{11/24},
					{31/66}/{1/2},
					{7/11}/{31/48},
					{43/66}/{2/3},
					{91/132}/{17/24},
					{95/132}/{3/4},
					{9/11}/{5/6},
					{223/264}/{7/8},
					{10/11}/{15/16},
					{21/22}/{1}} {
					\fill[black] ({\xa*\lengthLuroth},{\ya*\lengthLuroth}) rectangle ({\xb*\lengthLuroth},{\yb*\lengthLuroth});
				}
			}
			\draw[black] (0,0) rectangle ++ (\lengthLuroth,\lengthLuroth);
		\end{tikzpicture}
		\caption{SCC $L_{\leq 3}$ Level 3 product with SCC $L_{\leq 4}$ Level 4 on $[0,1]^2$}
		\label{fig: L_3 L_4}
	\end{figure}
	
	\begin{center}
		\rotatebox{270}{%
			\begin{minipage}{.99\textheight}
				\centering
				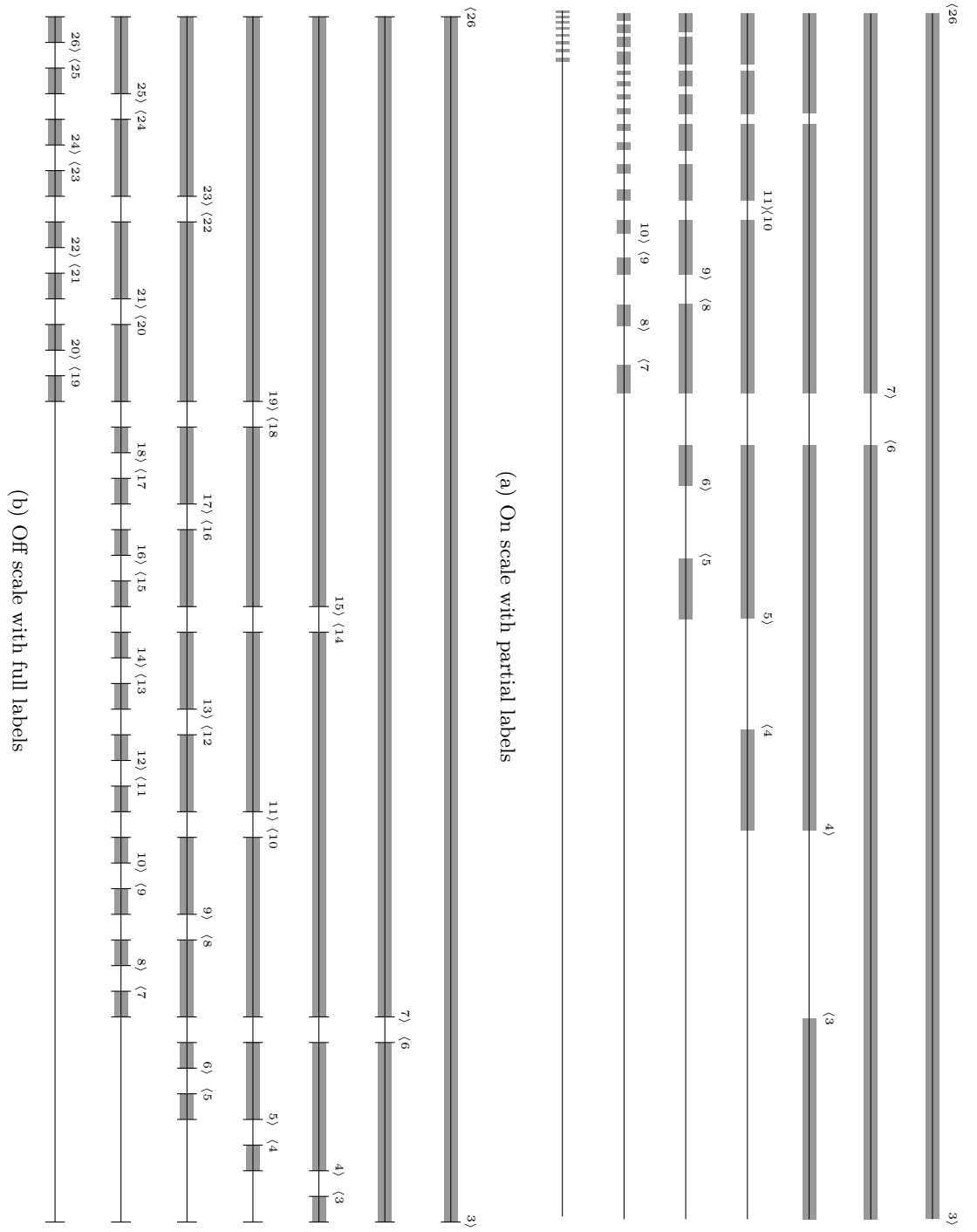
\begin{figure}[H]
					\centering
					\begin{subfigure}[b]{.98\textheight}
						\centering {\tiny 
							
							\begin{tikzpicture}
								\useasboundingbox ({f(26,26)},-0.45) rectangle ({f(3,3)},0.45);
								\fill [mycolour1] ({f(26,26)},-0.1) rectangle ({f( 3, 3)},0.1); 
								\draw ({f( 3, 3)},0)--({f(26,26)},0);
								\node at ({f( 3, 3)},.3) {$ 3\rangle$};
								\node at ({f(26,26)},.3) {$\langle26$};
							\end{tikzpicture}
							
							\begin{tikzpicture}
								\useasboundingbox ({f(26,26)},-0.45) rectangle ({f(3,3)},0.45);
								\fill [mycolour1] ({f(26,26)},-0.1) rectangle ({f( 7, 3)},0.1); 
								\fill [mycolour1] ({f( 6,26)},-0.1) rectangle ({f( 3, 3)},0.1); 
								\draw ({f( 3, 3)},0)--({f(26,26)},0);
								\node at ({f( 7, 3)},.3) {$ 7\rangle$};
								\node at ({f( 6,26)},.3) {$\langle 6$};
							\end{tikzpicture}
							
							\begin{tikzpicture}
								\useasboundingbox ({f(26,26)},-0.45) rectangle ({f(3,3)},0.45);
								\fill [mycolour1] ({f(26,26)},-0.1) rectangle ({f(15, 3)},0.1);
								\fill [mycolour1] ({f(14,26)},-0.1) rectangle ({f( 7, 3)},0.1);
								\fill [mycolour1] ({f( 6,26)},-0.1) rectangle ({f( 4, 3)},0.1);
								\fill [mycolour1] ({f( 3,26)},-0.1) rectangle ({f( 3, 3)},0.1);
								\draw ({f( 3, 3)},0)--({f(26,26)},0);
								\node at ({f( 4, 3)},.3) {$ 4\rangle$};
								\node at ({f( 3,26)},.3) {$\langle 3$};
							\end{tikzpicture}
							
							\begin{tikzpicture}
								\useasboundingbox ({f(26,26)},-0.45) rectangle ({f(3,3)},0.45);
								\fill [mycolour1] ({f(26,26)},-0.1) rectangle ({f(19, 3)},0.1);
								\fill [mycolour1] ({f(18,26)},-0.1) rectangle ({f(15, 3)},0.1);
								\fill [mycolour1] ({f(14,26)},-0.1) rectangle ({f(11, 3)},0.1);
								\fill [mycolour1] ({f(10,26)},-0.1) rectangle ({f( 7, 3)},0.1);
								\fill [mycolour1] ({f( 6,26)},-0.1) rectangle ({f( 5, 3)},0.1);
								\fill [mycolour1] ({f( 4,26)},-0.1) rectangle ({f( 4, 3)},0.1);
								\draw ({f( 3, 3)},0)--({f(26,26)},0);
								\node at ({f(11, 3)},.3) {$11\rangle$};
								\node at ({f(10,26)},.3) {$\langle10$};
								\node at ({f( 5, 3)},.3) {$ 5\rangle$};
								\node at ({f( 4,26)},.3) {$\langle 4$};
							\end{tikzpicture}
							
							\begin{tikzpicture}
								\useasboundingbox ({f(26,26)},-0.45) rectangle ({f(3,3)},0.45);
								\fill [mycolour1] ({f(26,26)},-0.1) rectangle ({f(23, 3)},0.1);
								\fill [mycolour1] ({f(22,26)},-0.1) rectangle ({f(19, 3)},0.1);
								\fill [mycolour1] ({f(18,26)},-0.1) rectangle ({f(17, 3)},0.1);
								\fill [mycolour1] ({f(16,26)},-0.1) rectangle ({f(15, 3)},0.1);
								\fill [mycolour1] ({f(14,26)},-0.1) rectangle ({f(13, 3)},0.1);
								\fill [mycolour1] ({f(12,26)},-0.1) rectangle ({f(11, 3)},0.1);
								\fill [mycolour1] ({f(10,26)},-0.1) rectangle ({f( 9, 3)},0.1);
								\fill [mycolour1] ({f( 8,26)},-0.1) rectangle ({f( 7, 3)},0.1);
								\fill [mycolour1] ({f( 6,26)},-0.1) rectangle ({f( 6, 3)},0.1);
								\fill [mycolour1] ({f( 5,26)},-0.1) rectangle ({f( 5, 3)},0.1);
								\draw ({f( 3, 3)},0)--({f(26,26)},0);
								\node at ({f( 9, 3)},.3) {$ 9\rangle$};
								\node at ({f( 8,26)},.3) {$\langle 8$};
								\node at ({f( 6, 3)},.3) {$ 6\rangle$};
								\node at ({f( 5,26)},.3) {$\langle 5$};
							\end{tikzpicture}
							
							\begin{tikzpicture}
								\useasboundingbox ({f(26,26)},-0.45) rectangle ({f(3,3)},0.45);
								\fill [mycolour1] ({f(26,26)},-0.1) rectangle ({f(25, 3)},0.1);
								\fill [mycolour1] ({f(24,26)},-0.1) rectangle ({f(23, 3)},0.1);
								\fill [mycolour1] ({f(22,26)},-0.1) rectangle ({f(21, 3)},0.1);
								\fill [mycolour1] ({f(20,26)},-0.1) rectangle ({f(19, 3)},0.1);
								\fill [mycolour1] ({f(18,26)},-0.1) rectangle ({f(18, 3)},0.1);
								\fill [mycolour1] ({f(17,26)},-0.1) rectangle ({f(17, 3)},0.1);
								\fill [mycolour1] ({f(16,26)},-0.1) rectangle ({f(16, 3)},0.1);
								\fill [mycolour1] ({f(15,26)},-0.1) rectangle ({f(15, 3)},0.1);
								\fill [mycolour1] ({f(14,26)},-0.1) rectangle ({f(14, 3)},0.1);
								\fill [mycolour1] ({f(13,26)},-0.1) rectangle ({f(13, 3)},0.1);
								\fill [mycolour1] ({f(12,26)},-0.1) rectangle ({f(12, 3)},0.1);
								\fill [mycolour1] ({f(11,26)},-0.1) rectangle ({f(11, 3)},0.1);
								\fill [mycolour1] ({f(10,26)},-0.1) rectangle ({f(10, 3)},0.1);
								\fill [mycolour1] ({f( 9,26)},-0.1) rectangle ({f( 9, 3)},0.1);
								\fill [mycolour1] ({f( 8,26)},-0.1) rectangle ({f( 8, 3)},0.1);
								\fill [mycolour1] ({f( 7,26)},-0.1) rectangle ({f( 7, 3)},0.1);
								\draw ({f( 3, 3)},0)--({f(26,26)},0);
								\node at ({f(10, 3)},.3) {$10\rangle$};
								\node at ({f( 9,26)},.3) {$\langle 9$};
								\node at ({f( 8, 3)},.3) {$ 8\rangle$};
								\node at ({f( 7,26)},.3) {$\langle 7$};
							\end{tikzpicture}
							
							\begin{tikzpicture}
								\useasboundingbox ({f(26,26)},-0.45) rectangle ({f(3,3)},0.45);
								\fill [mycolour1] ({f(26,26)},-0.1) rectangle ({f(26, 3)},0.1);
								\fill [mycolour1] ({f(25,26)},-0.1) rectangle ({f(25, 3)},0.1);
								\fill [mycolour1] ({f(24,26)},-0.1) rectangle ({f(24, 3)},0.1);
								\fill [mycolour1] ({f(23,26)},-0.1) rectangle ({f(23, 3)},0.1);
								\fill [mycolour1] ({f(22,26)},-0.1) rectangle ({f(22, 3)},0.1);
								\fill [mycolour1] ({f(21,26)},-0.1) rectangle ({f(21, 3)},0.1);
								\fill [mycolour1] ({f(20,26)},-0.1) rectangle ({f(20, 3)},0.1);
								\fill [mycolour1] ({f(19,26)},-0.1) rectangle ({f(19, 3)},0.1);
								\draw ({f( 3, 3)},0)--({f(26,26)},0);
							\end{tikzpicture}
						}
						\caption{On scale with partial labels}
					\end{subfigure}
					\hfill\vfill
					\begin{subfigure}[b]{0.98\textheight}
						\centering {\tiny 
							
							\begin{tikzpicture}
								
								\interval{26}{ 3}{0}{1}{1}
								
								\interval{26}{ 7}{1}{0}{1}
								\interval{ 6}{ 3}{1}{1}{0}
								
								\interval{26}{15}{2}{0}{1}
								\interval{14}{ 7}{2}{1}{0}
								\interval{ 6}{ 4}{2}{0}{1}
								\interval{ 3}{ 3}{2}{1}{0}
								
								\interval{26}{19}{3}{0}{1}
								\interval{18}{15}{3}{1}{0}
								\interval{14}{11}{3}{0}{1}
								\interval{10}{ 7}{3}{1}{0}
								\interval{ 6}{ 5}{3}{0}{1}
								\interval{ 4}{ 4}{3}{1}{0}
								
								\interval{26}{23}{4}{0}{1}
								\interval{22}{19}{4}{1}{0}
								\interval{18}{17}{4}{0}{1}
								\interval{16}{15}{4}{1}{0}
								\interval{14}{13}{4}{0}{1}
								\interval{12}{11}{4}{1}{0}
								\interval{10}{ 9}{4}{0}{1}
								\interval{ 8}{ 7}{4}{1}{0}
								\interval{ 6}{ 6}{4}{0}{1}
								\interval{ 5}{ 5}{4}{1}{0}
								
								\interval{26}{25}{5}{0}{1}
								\interval{24}{23}{5}{1}{0}
								\interval{22}{21}{5}{0}{1}
								\interval{20}{19}{5}{1}{0}
								\interval{18}{18}{5}{0}{1}
								\interval{17}{17}{5}{1}{0}
								\interval{16}{16}{5}{0}{1}
								\interval{15}{15}{5}{1}{0}
								\interval{14}{14}{5}{0}{1}
								\interval{13}{13}{5}{1}{0}
								\interval{12}{12}{5}{0}{1}
								\interval{11}{11}{5}{1}{0}
								\interval{10}{10}{5}{0}{1}
								\interval{ 9}{ 9}{5}{1}{0}
								\interval{ 8}{ 8}{5}{0}{1}
								\interval{ 7}{ 7}{5}{1}{0}
								
								\interval{26}{26}{6}{0}{1}
								\interval{25}{25}{6}{1}{0}
								\interval{24}{24}{6}{0}{1}
								\interval{23}{23}{6}{1}{0}
								\interval{22}{22}{6}{0}{1}
								\interval{21}{21}{6}{1}{0}
								\interval{20}{20}{6}{0}{1}
								\interval{19}{19}{6}{1}{0}
								
								\draw ({-0.02*\textheight},-0)--({-0.96\textheight},-0);
								\draw ({-0.02*\textheight},-1)--({-0.96\textheight},-1);
								\draw ({-0.02*\textheight},-2)--({-0.96\textheight},-2);
								\draw ({-0.02*\textheight},-3)--({-0.96\textheight},-3);
								\draw ({-0.02*\textheight},-4)--({-0.96\textheight},-4);
								\draw ({-0.02*\textheight},-5)--({-0.96\textheight},-5);
								\draw ({-0.02*\textheight},-6)--({-0.96\textheight},-6);
								\draw ({-0.02*\textheight},-3-.15)--({-0.02\textheight},-3+.15);
								\draw ({-0.02*\textheight},-4-.15)--({-0.02\textheight},-4+.15);
								\draw ({-0.02*\textheight},-5-.15)--({-0.02\textheight},-5+.15);
								\draw ({-0.02*\textheight},-6-.15)--({-0.02\textheight},-6+.15);
								
							\end{tikzpicture}
						}
						\caption{Off scale with full labels}
					\end{subfigure}
					\caption{An Unordered Construction of $L_{3,26}$}
					\label{fig: L_3_26}
				\end{figure}
			\end{minipage}
		}
	\end{center}
	
	%References
	\bibliographystyle{abbrv} 
	\bibliography{references.bib}

\end{document}